\documentclass[a4paper,11pt]{amsart}

\usepackage[left=2.3cm,right=2.3cm,top=3.5cm,bottom=3cm]{geometry}

\usepackage{amsmath,amssymb,amsthm,amscd,amsfonts}
\usepackage{mathrsfs}
\usepackage{latexsym}
\usepackage[all]{xy}
\usepackage{verbatim}
\usepackage[usenames, dvipsnames]{color}
\usepackage{MnSymbol}
\usepackage{multirow}
\usepackage{url}
\allowdisplaybreaks

\usepackage{hyperref}
\hypersetup{colorlinks,breaklinks,
            linkcolor=MidnightBlue,urlcolor=MidnightBlue,
            anchorcolor=MidnightBlue,citecolor=MidnightBlue}

\newcommand{\N}{\mathbb{N}}
\newcommand{\Z}{\mathbb{Z}}
\newcommand{\Q}{\mathbb{Q}}
\newcommand{\C}{\mathbb{C}}
\newcommand{\F}{\mathbb{F}}
\renewcommand{\P}{\mathbb{P}} 
\newcommand{\G}{\Gamma} 
\newcommand{\T}{\mathbf{T}} 
\newcommand{\Tr}{\mathcal{T}} 
\renewcommand{\c}{\mathbf{c}} 
\newcommand{\I}{\mathcal{I}} 
\newcommand{\g}{\gamma}
\renewcommand{\a}{\alpha}
\renewcommand{\b}{\beta}

\newcommand{\e}{\overline{e}}
\newcommand{\n}{\mathfrak{n}}

\renewcommand{\O}{\Omega}
\renewcommand{\o}{\omega}
\newcommand{\Hom}{\mathrm{Hom}}
\newcommand{\mf}{\,|_{k,m}}

\newtheorem{thm}{Theorem}[section]
\newtheorem{prop}[thm]{Proposition}
\newtheorem{lem}[thm]{Lemma}
\newtheorem{cor}[thm]{Corollary}
\newtheorem{defin}[thm]{Definition}
\newtheorem{rem}[thm]{Remark}
\newtheorem{exe}[thm]{Example}
\newtheorem{conj}[thm]{Conjecture}

\renewcommand{\matrix}[4]{\left(\begin{array}{cc} {#1} & {#2} \\ {#3} & {#4} \end{array}\right)}

\title{On the Atkin $U_t$-operator for $\Gamma_0(t)$-invariant Drinfeld cusp forms}

\author{Andrea Bandini}

\address{{\sc Andrea Bandini}: Universit\`a degli Studi di Parma \\
   Dipartimento di Scienze Matematiche, Fisiche e Informatiche\\
   Parco Area delle Scienze, 53/A \\
   43124 Parma - Italy
   }

\email{andrea.bandini@unipr.it}

\author{Maria Valentino}\thanks{M. Valentino is supported by an outgoing Marie-Curie fellowship of INdAM}
   \address{{\sc Maria Valentino}: Universit\`a degli Studi di Parma \\
   Dipartimento di Scienze Matematiche, Fisiche e Informatiche\\
   Parco Area delle Scienze, 53/A \\
   43124 Parma - Italy
   }

   \email{maria.valentino@unipr.it}



\begin{document}

\begin{abstract}
We study the diagonalizability of the Atkin $U_t$-operator acting on Drinfeld cusp forms for $\G_0(t)$: starting with the slopes
of eigenvalues and then moving to the space of cusp forms for $\G_1(t)$ to use Teitelbaum's interpretation as harmonic cocycles
which makes computations more explicit. We prove $U_t$ is diagonalizable in odd characteristic for (relatively) small weights
and explicitly compute the eigenvalues. In even characteristic we show that it is not diagonalizable when the weight is odd
(except for the trivial cases) and prove some cases of non diagonalizability in even weight as well.
We also formulate a few conjectures, supported by numerical search, about diagonalizability of $U_t$ and the slopes
of its eigenforms.
\end{abstract}

\maketitle

\section{Introduction}
Let $N,k\in\Z_{\geqslant 0}$ and denote by $S_k(N)$ the $\C$-vector space of cuspidal modular forms of level $N$ and weight $k$.
Hecke operators $T_n$, $n\geqslant 1$, are defined on $S_k(N)$ and when $p|N$, $T_p$ is also known as the {\em Atkin},
or {\em Atkin-Lehner $U_p$-operator}.\\
The $U_p$-operator is well known in literature for detecting those modular forms which belong to a $p$-adic fa\-mi\-ly. Let $f$ be an
eigenform of $U_p$ and $\a$ its eigenvalue. The {\em slope} of $f$ is defined to be the $p$-adic valuation of $\a$,
say $v_p(\a)$. Coleman in \cite{C} proved the existence
of a lot of $p$-adic families thanks to the fact that {\em overconvergent} modular forms of small slope are classical.
Moreover, eigenforms of slope zero
are the so-called {\em ordinary} forms which play a crucial role in Hida theory (\cite{H1}, \cite{H2}). \\
When $p$ is a prime number not dividing $N$, using Petersson inner product, the action of $T_p$ is semisimple on cusp forms.
This is no longer true for $U_p$ which fails to be diagonalizable. Some results on its semisimplicity are obtained in \cite{CE}.

In this work we shall address the problem of the diagonalizability of the analogue of $U_p$ in the realm of global
function fields.\\
When dealing with a field in positive characteristic there are two different translations of classical modular forms theory:
automorphic forms and Drinfeld modular forms.
The first ones are functions on adelic groups with values in fields of characteristic zero, while the latter are functions
on a rigid analytic space with values in a field of positive characteristic.\\
In this paper we will deal with Drinfeld modular forms which are equipped with a Hecke action as well.
In this context, the lack of an adequate analogous of Petersson
inner product leaves open the question about the diagonalizability  of the Hecke operators.
For the analogue of the $T_p$ some partial answers were given by Li and Meemark in \cite{LM} and B\"ockle and Pink in \cite{B}:
in contrast to the characteristic $0$ case these operators are not always diagonalizable. We shall focus on the diagonalizability
of the operator $U_t$ on forms of level $t$ and, in particular, on the slopes of its eigenforms.

We now introduce some notations, but for details and precise definitions of all objects appearing in this introduction the reader is
referred to Section \ref{SecNotations}.

Let $F$ be the rational function field $F =\mathbb{F}_{q}(t)$, with $q=p^r$ where $p\in \Z$ is a prime. Denote by $\pi$
the prime corresponding to $1/t$ and let $F_\infty = \F_q((1/t))$ be the
completion of $F$ at $\pi$ and let $\C_\infty$ denote the completion of an algebraic closure of $F_\infty$.
The Drinfeld upper half plane is $\O:= \P^1(\C_\infty) - \P^1(F_\infty)$. The $\C_\infty$-vector space of Drinfeld modular forms of weight
$k\geqslant 0$ and type $m\in \Z$ for a congruence subgroup $\G$ is denoted by $M_{k,m}(\G)$. The corresponding space of cusp forms is
indicated by $S^1_{k,m}(\G)$ and that of double cusp forms by $S^2_{k,m}(\G)$.

When dealing with number fields there is a direct relation between the Hecke
eigenvalues and the Fourier coefficients of a given modular form. In the function field setting,
even if it is still possible to associate with every Drinfeld modular
form a power series expansion with respect to a canonical uniformizer $u$ at the cusp at infinity (see \cite{G2}),
the action of the Hecke operators on expansions is not well understood and difficult to handle.\\
In order to avoid this problem we exploited a different reformulation
of Drinfeld cusp forms. The Bruhat-Tits tree
$\Tr$ (see Section \ref{SecTree}) is a combinatorial counterpart of $\O$ and cusp forms have a reinterpretation as $\G$-invariant
harmonic cocycles (Section \ref{SecIsomModFrmHarCoc})
\[  S^1_{k,m}(\G)\simeq C^{har}_{k,m}(\G)\,. \]
We are mainly interested in the case $\G=\G_0(t)$ (or $\G_0(\mathfrak{m})$ in general), but computations are
more feasible for $\G=\G_1(t)$. Since a Hecke action can be carried out on harmonic cocycles as well,
we will use this combinatorial interpretation along with a detailed description of a fundamental domain for $\G=\G_1(t)$
(obtained from \cite{GN}) to get the matrices corresponding to our operator $U_t$: we will find out that the
coefficients of these matrices are binomial coefficients depending only on the weight of the space of cusp
forms and the characteristic of the field involved. \\
An explicit description of the subspace $S^1_{k,m}(\G_0(t))$ inside $S^1_k(\G_1(t))$ (we do not mention the type here
because it is not relevant since all matrices in $\G_1(t)$ have trivial determinant) and a careful study of these binomial
coefficients will allow us to study the diagonalizability of $U_t$ on $\G_0(t)$-invariant cusp forms in some nontrivial cases
or, at least, to compute the slopes of its eigenforms. We shall address the diagonalizability on the whole $\Gamma_1(t)$-invariant space (at least in small weights) in 
\cite{BV}.

The paper is organized as follows.

\noindent In Section \ref{SecNotations} we set the notations and recall the main objects we shall work with. Moreover,
we recall the action of $U_t$ operators on harmonic cocycles (Section \ref{SecHecke}).

\noindent In Section \ref{SecSlopes} we study slopes for eigenforms for the action of $U_t$ on $S^1_{k,m}(\G_0(t))$,
by dividing cuspidal forms of level $t$ into oldforms (i.e., arising from $S^1_{k,m}(\G_0(1))\,$) and newforms.
Here we are able to say something on diagonalizability only in the very special case of odd weight $k$ in even characteristic.
Indeed

\begin{thm}\label{IntroThm1}
Let $k$ be odd and $q$ be even, then the action of $U_t$ on $S^1_{k,m}(\G_0(t))$ is not diagonalizable.
\end{thm}

\noindent The proof is based on the presence of an inseparable eigenvalue (see Sections \ref{SecSlopesNew} and \ref{SecSlopInsep})
and on Theorem \ref{ThmNonDiag2} (which holds for any characteristic \footnote{But we believe its hypotheses are very unlikely
to hold, see Remark \ref{RemNonDiag2}.}) for the remaining case.

\noindent In Section \ref{SecGamma1} we move to $\G=\G_1(t)$ and compute the matrix associated with
$U_t$ with respect to the basis $\mathcal{B}^1_k(\G_1(t)):=\{\c_j(\e)\,,\,0\leqslant j\leqslant k-2\}$
of $S^1_k(\G_1(t))$ (see Section \ref{UBasis}), where $\c_j$ are harmonic cocycles and $\e$ is a particular edge
of the fundamental domain. The crucial formula is
\begin{align}\label{Ucj}
U_t(\c_j(\e)) & = -(-t)^{j+1} \binom{k-2-j}{j} \c_j(\e) -t^{j+1}\sum_{h\neq 0}\left[ \binom{k-2-j-h(q-1)}{-h(q-1)} \right.\\
\ &\left. + (-1)^{j+1} \binom{k-2-j-h(q-1)}{j} \right] \c_{j+h(q-1)}(\e) \,. \nonumber
\end{align}
As it is easy to see from equation \eqref{Ucj}, the $\c_j$ can be divided into classes modulo $q-1$
and these classes are stable under the action of $U_t$. Thus, we can write the matrix associated with
$U_t$ in (at most) $q-1$ blocks and $U_t$ is diagonalizable if and only if each block is.
We shall denote by $C_j$ the class of $\c_j$, i.e.,  $C_j = \{\c_j, \c_{j+(q-1)}, \dots \}$ and by
$M_j$ the associated matrix. We detect the classes $C_j$ associated with $\G_0(t)$-invariant cuspforms
(see Section \ref{SecSpBlocks}) and then use formula \eqref{Ucj} to find coefficients for $M_j$. Those
matrices turn out to have certain symmetries which are summarized in Section \ref{SecSymmetry}.

In Section \ref{Char2Sec} we study the matrices $M_j$ associated to $S^1_{k,m}(\G_0(t))$ in some specific cases:
when the dimension of $M_j$ is relatively small with respect to $q$, the matrix has a form which allows us to decide about
its diagonalizability or, at least, to easily compute the slopes of its eigenforms. In particular, in the nontrivial
cases where the dimension of $M_j$ is at least 2, Theorems \ref{ThmAntidiagonal} and \ref{Thmj=0} and
Section \ref{Sec=j+2&j+3} show that\begin{itemize}
\item if $\dim(M_j)\leqslant j+1$, then $M_j$ is antidiagonal and it is diagonalizable if and only if $q$ is odd;
\item if $j=0$ and $\dim(M_0)\leqslant q+2$, then $M_j$ is diagonalizable unless $q$ is even and $\dim(M_0)\geqslant 4$;
\item if $\dim(M_j)\leqslant 4$ and $q$ is odd, then $Mj$ is diagonalizable.
\end{itemize}

As $\dim(M_j)$ grows it becomes harder to find a pattern for the matrices $M_j$, so we continued our investigation via some
computer search. In Section \ref{SecTables} we provide the links to the files containing the outcome of our computations
on characteristic polynomials and slopes for the action of $U_t$ on $S^1_{k,m}(\G_0(t))$ (still seen as a
subspace of $S^1_k(\G_1(t))\,$) and some speculations on what might be worth of more investigation in the future.\bigskip

\noindent{\bf Acknowledgements.} The first version of this paper included many more computations (and, unfortunately,
much less insight): we are extremely grateful to Gebhard B\"ockle for pointing out the possibility
of pursuing the path of the Gouv\^ea-Mazur papers in our setting and for useful conversations. 
We also thank Rudolph Perkins for helpful suggestions and for pointing out the paper \cite{Vi}.
Deep thanks are due to Kevin Buzzard for providing us with enlightening hints and ideas which have been crucial
for the development of the whole paper.

\section{Setting and notations}\label{SecNotations}

Here we set the notations and collect all basic facts and technical tools we shall need throughout the paper.

Let $F$ be the global function field $F:=\F_q(t)$, where $q$ is a power of a fixed prime $p\in\Z$, and denote by
$A:=\F_q[t]$ its ring of integers. Let $F_\infty$ be the completion of $F$ at $\pi:=\frac{1}{t}$ with ring of integers $A_\infty$
and denote by $\C_\infty$ the completion of an algebraic closure of $F_\infty$.\\
The {\em Drinfeld upper half-plane} is the set $\O:=\P^1(\C_\infty) - \P^1(F_\infty)$
together with a structure of rigid analytic space (see \cite{FvdP}).

\subsection{The Bruhat-Tits tree}\label{SecTree}
The Drinfeld's upper half plane has a combinatorial counterpart, the {\em Bruhat-Tits tree} $\Tr$ of $GL_2(F_\infty)$, which
we shall describe briefly here. For more details the reader is referred to \cite{G3}, \cite{G4} and \cite{S1}.\\
The tree $\Tr$ is a $(q+1)$-regular tree on which $GL_2(F_\infty)$ acts transitively. Let us denote by $Z(F_\infty)$ the scalar
matrices of $GL_2(F_\infty)$ and by $\I(F_\infty)$ the {\em Iwahori subgroup}, i.e.,
\[ \I(F_\infty)=\left\{\matrix{a}{b}{c}{d} \in GL_2(A_\infty)\, : \, c\equiv 0 \pmod \pi \right\}\,. \]
Then the sets $X(\Tr)$ of vertices and $Y(\Tr)$ of oriented edges of $\Tr$ are given by
\begin{align*}
X(\Tr) & = GL_2(F_\infty)/Z(F_\infty)GL_2(A_\infty)\\
Y(\Tr) & = GL_2(F_\infty)/ Z(F_\infty)\I(F_\infty).
\end{align*}
The canonical map from $Y(\Tr)$ to $X(\Tr)$ associates with each oriented edge $e$ its origin $o(e)$
(the corresponding terminus will be denoted by $t(e)\,$). The edge $\overline{e}$ is $e$ with
reversed orientation.\\
A system of representatives of $X(\Tr)$ and $Y(\Tr)$ is
\begin{align*}
S_X & := \left\{ v_{i,u}=\matrix {\pi^{-i}}{u}{0}{1} \, :\, i\in\Z\,,\ u\in F_\infty/\pi^{-i} A_\infty \right\} \\
S_Y & := S_X \bigcup S_X\matrix {0}{1}{\pi}{0}\, .
\end{align*} Two infinite paths in $\Tr$ are considered equivalent if they differ at finitely many edges. An {\em end} is an equivalence
class of infinite paths. There is a $GL_2(F_\infty)$-equivariant bijection between the ends of $\Tr$ and $\P^1(F_\infty)$.
An end is called {\em rational} if it corresponds to an element in $\P^1(F)$ under the above bijection.
Moreover, for any arithmetic subgroup $\G$ of $GL_2(A)$, the elements of $\G\backslash \P^1(F)$ are in bijection with
the ends of $\G\backslash \Tr$ (see \cite[Proposition 3.19]{B} and \cite[Lecture 7, Proposition 3.2]{GPRV})
and they are called the {\em cusps} of $\G$.\\
Following Serre \cite[pag 132]{S1}, we call a vertex or an edge {\em $\G$-stable} if its stabilizer in $\G$ is trivial
and {\em $\G$-unstable} otherwise.

\subsection{Drinfeld modular forms}\label{SecDrinfModForms}
The group $GL_2(F_\infty)$ acts on $\Omega$ by fractional linear transformation
\[  \left( \begin{array}{cc}
a & b  \\
c & d
\end{array} \right)(z)= \frac{az+b}{cz+d}\,. \]
Let $\G$ be an arithmetic subgroup of $GL_2(A)$. It has finitely many cusps, represented by $\G\backslash \P^1(F)$
as we have seen above. For $\g =\bigl(\begin{smallmatrix}a&b\\c&d\end{smallmatrix}\bigr)\in GL_2(F_\infty)$, $k,m \in \Z$
and $\varphi:\O\to \C_\infty$, we define
\begin{align*}   (\varphi \mf \g)(z) := \varphi(\g z)(\det \g)^m(cz+d)^{-k}. \end{align*}
A rigid analytic function $\varphi:\O\to \C_\infty$ is called a {\em Drinfeld modular function of weight $k$ and type $m$ for $\G$} if
\begin{equation}\label{Mod} (\varphi \mf \g )(z) =\varphi(z)\ \ \forall \g\in\G\,.  \end{equation}

\begin{defin}
A Drinfeld modular function $\varphi$ of weight $k\geqslant 0$ and type $m$ for $\G$ is called a {\em Drinfeld modular form} if
$\varphi$ is holomorphic at all cusps.\\
A Drinfeld modular form $\varphi$ is called a {\em cusp form}(resp. {\em double cusp form}) if  it vanishes at all
cusps to the order at least 1 (resp. to the order at least 2).\\
The space of Drinfeld modular forms
of weight $k$ and type $m$ for $\G$ will be denoted by $M_{k,m}(\G)$. The subspace of cuspidal modular forms (resp. doubly
cuspidal) is denoted by $S^1_{k,m}(\G)$ (resp. $S^2_{k,m}(\G))$.
\end{defin}

\noindent Note that weight and type are not independent of each other. Taking the matrix
$\g=\bigl(\begin{smallmatrix}a&0\\0&a\end{smallmatrix}\bigr)\in\G$
for some $a\in\F_q^*$ of multiplicative order $\ell$ and substituting it in equation \eqref{Mod}, one sees that,
if $k\not\equiv 2m \pmod \ell$, then $M_{k,m}(\G)=0$.\\
Moreover, if $\G$ is such that all its elements have determinant equal to one it follows that
the type does not play a role. If this is the case, for fixed $k$ all $M_{k,m}(\G)$ are isomorphic (the same holds for
$S^1_{k,m}(\G)$ and $S^2_{k,m}(\G)$) and we will simply denote them by $M_{k}(\G)$ (resp. $S^1_{k}(\G)$ and $S^2_{k}(\G)$).

\noindent All $M_k(\G)$ and $S^i_k(\G)$ $i=1,2$ are finite dimensional $\C_\infty$-vector spaces. In particular
when $\G$ has no prime to $p$ torsion, we have the following

\begin{thm}\label{DimTorFree}
Let $g(\G)$ denote the genus of $\G\backslash \Tr$ and $h(\G)$ the number of its cusps. If $\G$ has no prime to $p$ torsion
 and $\det(\G)=1$, then
\begin{itemize}
\item[{\bf 1.}] {$\dim_{\C_\infty} S^1_k(\G)=(k-1)(g(\G)+h(\G)-1)$;}
\item[{\bf 2.}] {$\dim_{\C_\infty} S^2_k(\G)= \left\{ \begin{array}{cr}
g(\G) & k=2 \\
\ & \\
(k-2)(g(\G)+h(\G)-1)+g(\G)-1 & k>2
\end{array}\right.$}
\end{itemize}
\end{thm}

\begin{proof}
See \cite[Proposition 5.4]{B} and \cite[Proposition 5.18]{B}.
\end{proof}
For an extensive calculation of the genus and the number of cusps for any $\G$ see \cite{GN}.

\subsection{Harmonic cocycles}\label{SecHarCoc}
For $k\geqslant 0$ and $m\in\Z$, let $V(k,m)$ be the $(k-1)$-dimensional vector space over $\C_\infty$ with basis
$\{X^jY^{k-2-j}: 0\leqslant j\leqslant k-2 \}$. The action of $\g=\matrix{a}{b}{c}{d} \in GL_2(F_\infty)$ on $V(k,m)$ is given by
\[ \g(X^jY^{k-2-j}) = \det(\g)^{m-1}(dX-bY)^j(-cX+aY)^{k-2-j}\quad {\rm for}\ 0\leqslant j\leqslant k-2\,.\]
For every $\o\in \Hom(V(k,m),\C_\infty)$ we have an induced action of $GL_2(F_\infty)$
\[ (\g\o)(X^jY^{k-2-j})=\det(\g)^{1-m}\o((aX+bY)^j(cX+dY)^{k-2-j})\quad {\rm for}\ 0\leqslant j\leqslant k-2\,. \]
\begin{defin}
A {\em harmonic cocycle of weight $k$ and type $m$ for $\G$} is a function $\c$ from the set of directed edges
of $\Tr$ to $\Hom(V(k,m),\C_\infty)$ satisfying:
\begin{itemize}
\item[{\bf 1.}] ({\em harmonicity}) for all vertices $v$ of $\Tr$,
\[  \sum_{t(e)= v}\c(e)=0 \]
where $e$ runs over all edges in $\Tr$ with terminal vertex $v$;
\item[{\bf 2.}] ({\em antisymmetry}) for all edges $e$ of $\Tr$, $\c(\overline{e})=-\c(e)$;
\item[{\bf 3.}] ({\em $\G$-equivariancy}) for all edges $e$ and elements $\g\in\G$, $\c(\g e)=\g(\c(e))$.
\end{itemize}
\end{defin}

\noindent The space of harmonic cocycles of weight $k$ and type $m$ for $\G$ will be denoted by $C^{har}_{k,m}(\G)$.\\
We point out that, as a consequence of \cite[Lemma 20]{T}, cocycles in $C^{har}_{k,m}(\G)$ are determined by their
values on the stable (non-oriented) edges of a fundamental domain.

\subsubsection{Cusp forms and harmonic cocycles}\label{SecIsomModFrmHarCoc}
In \cite{T}, Teitelbaum constructed the so-called ``residue map'' which allow us to us to interpret cusp forms as harmonic cocycles.
Indeed, it is proved in \cite[Theorem 16]{T} that this map is actually an isomorphism
\[  S^1_{k,m}(\G)\simeq C^{har}_{k,m}(\G)\,. \]
For more details on the subject the reader is referred to the original paper of Teitelbaum or to \cite[Section 5.2]{B} which
is full of details written in a more modern language.

\subsection{Hecke operators}\label{SecHecke}
Hecke operators on Drinfeld modular forms are formally defined using a double coset decomposition as showed in
\cite{A}. Here we will use an equivalent definition employing a simplified notations which will involve just
polynomials of $A$.\\
Let $\n$ be an ideal of $A$ and denote by $P_\n$ its monic generator. The {\em Hecke operator} $\T_\n$ acts
on $\varphi\in M_{k,m}(\G)$ in the following way
\[ \T_\n(\varphi)(z):=\sum_{\begin{subarray}{c} \alpha,\delta\ \mathrm{monic}\\
\beta\in A, \deg(\beta)<\deg(\delta)\\
\alpha\delta=P_\n, (\alpha)+(t)=A\end{subarray}}
( \varphi \mf \matrix{\alpha}{\beta}{0}{\delta} )(z) \,.\]
Note that when $\n$ is a prime ideal different from $(t)$ and such that $P_\n$ has degree one we find the usual definition
\[ \T_\n(\varphi)(z)=  ( \varphi \mf \matrix{P_\n}{0}{0}{1} ) (z) +
\sum_{\beta\in \F_q}  ( \varphi \mf \matrix{1}{\beta}{0}{P_\n} ) (z)\,. \]
In particular, for $\n=(t)$ the Atkin $U_t$-operator in our context is
\[ U_t(\varphi)(z):=\T_{(t)}(\varphi)(z)= t^{m-k} \sum_{\beta\in \F_q} \varphi\left(\frac{z+\beta}{t}\right)\,. \]
We would like to point out that some authors just declared $U_t$ to be the zero map (see, e.g., B\"ockle \cite[Definition 6.5]{B}),
while others (like Armana \cite{A} or Goss \cite{Go}) included a not trivial $U_t$-operator in the Hecke algebra
they worked with.

The residue map allows us to define a Hecke action on harmonic cocycles in the following way:
\begin{align*}
U_t(\c(e))=  \sum_{\beta\in \F_q}  \left(\begin{array}{cc}
1 & \beta  \\
0 & t
\end{array}\right)^{-1}\c\left( \left(\begin{array}{cc}
1 & \beta  \\
0 & t
\end{array}\right)e\right)
\end{align*}
(for details see formula (17) in \cite[Section 5.2]{B}). Note that this corresponds to the $U_t(f)$ above, but
the definition we shall use in Section \ref{SecGamma1} differs by a factor $t^{k-m}$, which we include there to eliminate
any reference to the type $m$ in the final formulas (this will interfere a bit with the slopes but not with diagonalizability).

\noindent We shall focus on the congruence groups $\G_0(\mathfrak{m})$ and $\G_1(\mathfrak{m})$ (where $\mathfrak{m}$ is
an ideal of $A$ and, almost always, it will simply be $t$ or $1$), where
\[ \G_0(\mathfrak{m})=\left\{ \left( \begin{array}{cc}
a & b  \\
c & d
\end{array} \right)\in GL_2(A): c\equiv 0 \pmod{\mathfrak{m}} \right\}\]
and
\[ \G_1(\mathfrak{m})=\left\{ \left( \begin{array}{cc}
a & b  \\
c & d
\end{array} \right)\in GL_2(A): a\equiv d\equiv 1\ \mathrm{and}\ c\equiv 0 \pmod{\mathfrak{m}} \right\}.\]

The main strategy is the following. As showed by Teitelbaum (\cite[Theorem 3]{T}), every harmonic cocycle vanishes on all
but finitely many edges of the fundamental domain.
So, we need to evaluate a harmonic cocycle just on stable edges, as already noted at the end of Section \ref{SecHarCoc},
and then extend to other (non-vanishing) edges by $\G$-equivariancy and harmonicity. In this way we shall obtain
the action of $U_t$ and then we will proceed with analyzing its diagonalizability and/or the slopes of its eigenforms.

\section{Slopes of Hecke eigenforms for $S^1_{k,m}(\G_0(t))$} \label{SecSlopes}
We shall consider two related problems together: the diagonalizability of the matrix associated to the action of
the Hecke operator $U_t$ on $S^1_{k,m}(\G_0(t))$ and the slopes of its eigenforms defined as

\begin{defin}\label{DefSlope}
Let $\varphi$ be an eigenform for $U_t$ of eigenvalue $\lambda\neq 0$, then the {\em slope} of $\varphi$ is $v_t(\lambda)$
(where $v_t$ is the natural $t$-adic valuation with $v_t(t)=1$).
\end{defin}

Knowing the slopes usually does not provide full information on the diagonalizability, unless (as we shall see) we
have a slope which yields an inseparable eigenvalue for $U_t$ and readily implies non-diagonalizability.
We shall start with some information on the slopes and then move to the space $S^1_k(\G_1(t))$ where
we can provide more explicit formulas for the matrices and a few more results on diagonalizability.

\subsection{The degeneracy maps} For any ideal $\mathfrak{m}$, we have two obvious maps which produce
{\em oldforms} in $S^1_{k,m}(\G_0(\mathfrak{m}t))$ (i.e., forms whose level is $\mathfrak{m}$ but which can be considered
also in $S^1_{k,m}(\G_0(\mathfrak{m}t))\,$):
\[ \delta_1\,,\delta_t: S^1_{k,m}(\G_0(\mathfrak{m})) \rightarrow S^1_{k,m}(\G_0(\mathfrak{m}t)) \]
\[ \delta_1\varphi:=\varphi \]
\[ \delta_t\varphi:=(\varphi\mf \matrix{t}{0}{0}{1})(z)\ {\rm ,i.e.,}\ (\delta_t\varphi)(z)=t^m\varphi(tz)\,. \]

\begin{prop}\label{PropDeltaInj}
Assume $t$ and $\infty$ do not divide $\mathfrak{m}$, then the map
\[ \delta : S_{k,m}^1(\G_0(\mathfrak{m}))\times S_{k,m}^1(\G_0(\mathfrak{m})) \longrightarrow S_{k,m}^1(\G_0(\mathfrak{m}t)) \]
\[ \delta(\varphi,\psi):=\delta_1\varphi+\delta_t\psi \]
is injective.
\end{prop}

\begin{proof}
Let $(\varphi, \psi)\in Ker(\delta)$, then
\[ \delta_1\varphi=-\delta_t\psi=-(\psi\mf\matrix{t}{0}{0}{1}) \,.\]
Therefore, for any $\gamma\in \Gamma_0(\mathfrak{m})$,
\begin{align}
(\delta_1\varphi\mf\matrix{t}{0}{0}{1}^{-1}\gamma \matrix{t}{0}{0}{1}) & =
-(\delta_t\psi\mf\matrix{t}{0}{0}{1}^{-1}\gamma \matrix{t}{0}{0}{1}) \nonumber \\
\ & = -(\psi\mf\gamma \matrix{t}{0}{0}{1}) = - (\psi\mf\matrix{t}{0}{0}{1}) \\
\ & =-\delta_t\psi=\delta_1\varphi \nonumber
\end{align}
i.e., $\varphi$ is invariant for $\matrix{t}{0}{0}{1}^{-1}\Gamma_0(\mathfrak{m}) \matrix{t}{0}{0}{1}$ as well.

\noindent We have been working on $\mathcal{T}$, the {\em Bruhat-Tits tree at $\infty$}, now we consider also $\mathcal{T}_t$
the {\em Bruhat-Tits tree at $t$} associated with $GL_2(F_t)$ ($F_t$ is the completion of $F$
at the prime $t$). Let $S:=\{\infty,t\}$ and put $A_S$ as the ring of elements regular outside $S$. Define
\[ G(\mathfrak{m}):=\left\{ \gamma=\matrix{a}{b}{c}{d}\in GL_2(A_S)\,:
\,\det(\gamma)\in \F_q^*\ {\rm and}\ c\equiv 0\pmod{\mathfrak{m}}\right\} \]
(see \cite[Section 2.1]{HL}).
By \cite[Corollary 2.2]{HL} $G(\mathfrak{m})$ acts transitively on both $\mathcal{T}$ and $\mathcal{T}_t$, hence
there is a single fundamental edge for the action of $G(\mathfrak{m})$.

\noindent Now consider a fundamental edge $e_t$ for $\mathcal{T}_t$ with terminus $w_t$ and origin
$v_t=\matrix{t}{0}{0}{1}^{-1}w_t$.
As remarked in \cite{HL} (first paragraph of page 162) the stabilizers of $e_t$, $w_t$ and $v_t$ in $G(\mathfrak{m})$ are
respectively
\[ Stab_{G(\mathfrak{m})}(e_t)=\Gamma_0(\mathfrak{m}t) \]
(denoted $\Gamma_0^\infty(\mathcal{m}t)$ in \cite{HL})
\[ Stab_{G(\mathfrak{m})}(w_t)=\Gamma_0(\mathfrak{m}) \]
\[ Stab_{G(\mathfrak{m})}(v_t)=\matrix{t}{0}{0}{1}^{-1}\Gamma_0(\mathfrak{m})\matrix{t}{0}{0}{1} \,.\]
Therefore \cite[Theorem 6, page 32]{S1} yields
\[ G(\mathfrak{m})= \Gamma_0(\mathfrak{m})*_{\Gamma_0(\mathfrak{m}t)}
\matrix{t}{0}{0}{1}^{-1}\Gamma_0(\mathfrak{m})\matrix{t}{0}{0}{1} \]
($*$ here is the {\em amalgamated product}).
\noindent Finally we have seen that $\varphi\in Ker(\delta)$ implies $\varphi$ is $\Gamma_0(\mathfrak{m})$ and
$\matrix{t}{0}{0}{1}^{-1}\Gamma_0(\mathfrak{m})\matrix{t}{0}{0}{1}$ invariant, so it is $G(\mathfrak{m})$ invariant.
Since $G(\mathfrak{m})$ acts transitively on $\mathcal{T}$, then $\varphi=0$ (and, consequently, $\psi=0$ as well).
\end{proof}

\subsection{Hecke operators} We recall the definition of Hecke operators, in particular the one associated to $t$ and
note that the definition depends on the level $\mathfrak{m}$ of the form we are considering (in particular on the
divisibility of $\mathfrak{m}$ by $t$). Let $\varphi\in S_{k,m}^1(\G_0(\mathfrak{m}))$, then
\[ U_t\varphi(z):=\left\{\begin{array}{ll}
\displaystyle{(\varphi\mf\matrix{t}{0}{0}{1})(z)+
\sum_{b\in\F_q} (\varphi\mf\matrix{1}{b}{0}{t})(z) } & {\rm if}\ t\nmid \mathfrak{m}  \\
\ & \\
\displaystyle{ \sum_{b\in\F_q} (\varphi\mf\matrix{1}{b}{0}{t})(z) } & {\rm if}\ t\mid \mathfrak{m}  \end{array} \right. \ .\]
Since our maps will move back and forth from level $\mathfrak{m}$ to level $\mathfrak{m}t$ (and often $\mathfrak{m}$
will simply be 1) we use different notations $\T_t$ for the Hecke operator of level $\mathfrak{m}$ (assuming $t\nmid \mathfrak{m}$)
and $U_t$ for the Hecke operator of level $\mathfrak{m}t$.

\subsection{Slopes of oldforms}\label{SecSlopesOld}
We now check the interaction between the degeneracy maps and the Hecke operator. Let $\varphi\in S_{k,m}^1(\G_0(\mathfrak{m}))$, then
(directly from the definition) one has
\begin{align*} \T_t\varphi & =(\varphi\mf\matrix{t}{0}{0}{1})+U_t(\delta_1\varphi) \\
\ & = \delta_t\varphi +U_t(\delta_1\varphi) \,,\end{align*}
hence the formula
\begin{equation}\label{Eq1Utdelta}
U_t(\delta_1\varphi)=\delta_1(\T_t\varphi)-\delta_t\varphi\,.
\end{equation}
For the other degeneracy map one has
\begin{align} \label{Eq2Utdelta}
U_t(\delta_t\varphi) & = \sum_{b\in\F_q} (\delta_t\varphi\mf\matrix{1}{b}{0}{t}) =
\sum_{b\in\F_q} (\varphi\mf\matrix{t}{0}{0}{1}\matrix{1}{b}{0}{t}) \nonumber \\
\ & = \sum_{b\in\F_q} (\varphi\mf\matrix{t}{tb}{0}{t}) = \sum_{b\in\F_q} (\varphi\mf\matrix{1}{b}{0}{1}\matrix{t}{0}{0}{t}) \\
\ & = t^{2m-k} \sum_{b\in\F_q} (\varphi\mf\matrix{1}{b}{0}{1}) =
t^{2m-k}\sum_{b\in\F_q} \varphi = qt^{2m-k} \varphi = 0 \,.\nonumber
\end{align}
In particular $U_t$ always has nontrivial kernel (hence nontrivial eigenspace for the eigenvalue 0),
because $U_t(Im(\delta_t))=0$.

\begin{prop}\label{PropOldSlopes}
We have an equality of sets
\[ \{ {\rm Eigenvalues\ of\ }{U_t}_{|Im(\delta)}\} = \{ {\rm Eigenvalues\ of\ }\T_t\}\cup\{0\}\,,\]
i.e, in terms of slopes
\[ \{ {\rm Slopes\ for\ }{U_t}_{|Im(\delta)}\} = \{ {\rm Slopes\ for\ }\T_t\} \,.\]
\end{prop}

\begin{proof}
We have just seen that 0 is always an eigenvalue for $U_t$. Now let $\delta(\varphi,\psi)$ be an (old) eigenform for $U_t$
of eigenvalue $\lambda\neq 0$, then
\[ \lambda\delta(\varphi,\psi)=U_t(\delta(\varphi,\psi))=\delta_1(\T_t\varphi)-\delta_t\varphi=\delta(\T_t\varphi,-\varphi)\,.\]
The injectivity of $\delta$ yields $\T_t\varphi=\lambda\varphi$. Vice versa it is easy to check that if $\T_t\varphi=\lambda\varphi$
for some $\lambda\neq 0$, then $U_t(\delta(\varphi,-\frac{1}{\lambda}\varphi))=\lambda\delta(\varphi,-\frac{1}{\lambda}\varphi)$.
\end{proof}

\begin{rem}\label{RemOld}
The behaviour of $U_t$ on oldforms is analogous to the classical case: the eigenvalues for ${U_t}_{|Im(\delta)}$ verify equations
like $X^2-\lambda X=0$ where $\lambda$ is a nonzero eigenvalue for $\T_t$ (in the classical case the equation was
$X^2-\lambda X+p^{k-1}=0$ which reduces to our one modulo $p$, see \cite[Section 4]{GM1}).
\end{rem}

If $\varphi\in Ker (\T_t)$, then $U_t$ acts on $\langle \delta_1\varphi,\delta_t\varphi\rangle$ via
$\matrix{0}{0}{-1}{0}$ which is not diagonalizable. It is not hard to prove that when $\T_t$ is diagonalizable, then
$U_t$ is diagonalizable on $Im (\delta)$ if and only if $Ker(\T_t)=\{0\}$. When $\mathfrak{m}=1$ we never found 0 among the eigenvalues
of $\T_t$, moreover our computations led us to believe that the operator $U_t$ is diagonalizable unless $q$ is even. We will
return on this topic in Section \ref{SecTables}.

\subsection{Fricke involution and trace maps}\label{SecTrace}
Classical newforms were linked to trace maps (see, e.g., \cite[Section 4]{GM1}); to provide information on newforms we need
to compute the trace (and the twisted trace as we shall soon see) of a form. Computations are feasible only from
$S^1_{k,m}(\G_0(t))$ to $S^1_{k,m}(\G_0(1))=S^1_{k,m}(GL_2(A))$, hence from now on we shall only consider the
case $\mathfrak{m}=1$.

\noindent Let
\[ \gamma_t:=\matrix{0}{-1}{t}{0} \]
(the {\em Fricke involution} as in \cite[end of Section 2.2]{Vi}).
To shorten notations we shall often use $\varphi^{Fr}$ to denote $(\varphi\mf \gamma_t)$. \\
It is easy to compute $(\varphi^{Fr})^{Fr}=(-1)^k t^{2m-k} \varphi$. When $q$ is even, $(-1)^k=1$ and, when $q$ is odd,
$k$ has to be even to have nonzero cuspidal forms. Hence, in any case, $(-1)^k=1$ and
\begin{equation}\label{EqFricke}
(\varphi^{Fr})^{Fr}=t^{2m-k} \varphi\,.
\end{equation}
For a modular form $\varphi$ of level $1$ one gets
\begin{equation}\label{Eqgammatdelta}
(\delta_1\varphi)^{Fr}= (\delta_t\varphi) \quad{\rm and}\quad
(\delta_t\varphi)^{Fr} = t^{2m-k}(\delta_1\varphi)\,.
\end{equation}
Therefore the effect of $U_t$ on those forms is given by
\[ U_t((\delta_1\varphi)^{Fr})=U_t(\delta_t\varphi)=0 \]
and
\[ U_t((\delta_t\varphi)^{Fr})=t^{2m-k}U_t(\delta_1\varphi)=t^{2m-k}[\delta_1(\T_t\varphi)-\delta_t\varphi] \,.\]
For the trace maps we use the system of representatives provided by \cite[Lemma 3.7]{Vi}, i.e.,
\[ R:=\left\{ {\bf Id}_2, \matrix{0}{-1}{1}{b}\ b\in \F_q \right\}\,.\]

\begin{defin}\label{DefTrace}
For any cuspidal form $\varphi$ of level $t$ define the {\em trace}
\[ Tr(\varphi):=\sum_{\gamma\in R} (\varphi\mf \gamma) \]
and the {\em twisted trace}
\[ Tr'(\varphi):=Tr(\varphi^{Fr}) =\sum_{\gamma\in R} (\varphi\mf \gamma_t\gamma) \,.\]
Both $Tr$ and $Tr'$ are maps from $S_{k,m}^1(\G_0(t))$ to $S_{k,m}^1(\G_0(1))$ (see \cite[Definition 3.5]{Vi}).
\end{defin}

\noindent The basic equation for the trace map is the following
\begin{align}\label{EqTr}
Tr(\varphi) & = \varphi+\sum_{b\in \F_q} \left( \varphi\mf \matrix{0}{-1}{1}{b}\right) \nonumber \\
\ & = \varphi+\sum_{b\in \F_q}
\left( \varphi\mf \matrix{0}{-1}{t}{0}\matrix{1}{b}{0}{t}\matrix{\frac{1}{t}}{0}{0}{\frac{1}{t}}\right) \\
\ & = \varphi+ t^{k-2m} U_t(\varphi\mf \gamma_t) = \varphi+ t^{k-2m} U_t(\varphi^{Fr})\,. \nonumber
\end{align}
From this and \eqref{EqFricke}, one readily obtains
\begin{equation}\label{EqTr'}
Tr'(\varphi)=Tr(\varphi^{Fr})=\varphi^{Fr}+t^{k-2m} U_t((\varphi^{Fr})^{Fr})=\varphi^{Fr}+U_t(\varphi)\,.
\end{equation}

\subsection{Slopes of newforms}\label{SecSlopesNew}
Newforms were defined in \cite[Section 4]{GM1} as forms in the $Ker(Tr)$ but in our setting one sees that,
since all representatives in $R$ are obviously in $GL_2(A)$ and have determinant 1,
\[ Tr(\delta_1\varphi)=\sum_{\gamma\in R} (\varphi\mf \gamma)= \sum_{\gamma\in R} \varphi =(q+1)\varphi=\varphi\,,\]
and
\begin{align}\label{Eq1Trdelta}
Tr(\delta_t\varphi) & = \sum_{\gamma\in R} (\varphi\mf \matrix{t}{0}{0}{1}\gamma) =
\delta_t\varphi + \sum_{b\in \F_q} (\varphi\mf \matrix{0}{-t}{1}{b}) \nonumber \\
\ & = \delta_t\varphi + \sum_{b\in \F_q} (\varphi\mf \matrix{0}{-1}{1}{0}\matrix{1}{b}{0}{t}) \\
\ & = (\varphi\mf \matrix{t}{0}{0}{1}) + \sum_{b\in \F_q} (\varphi\mf \matrix{1}{b}{0}{t}) = \T_t\varphi \,. \nonumber
\end{align}
Therefore for any $\T_t$-eigenform $\varphi$ (of level 1) of nonzero eigenvalue $\lambda$
\[ Tr(\delta(\varphi,-\frac{1}{\lambda}\varphi))= \varphi -\frac{1}{\lambda}\T_t\varphi=0\,.\]

\noindent A more suitable definition for {\em newforms} seems to involve $Ker(Tr')$. Moreover directly from the definition one has
that $\varphi\in Ker(Tr') \iff \varphi^{Fr}\in Ker(Tr)$, i.e., $Ker(Tr')=Ker(Tr)^{Fr}$ and we define the newforms as
$Ker(Tr)\cap Ker(Tr)^{Fr}$.

Let $\varphi$ (of level $t$) be a $U_t$-(new)eigenform of eigenvalue $\lambda$, then, by \eqref{EqTr'},
\[ \varphi^{Fr}=-U_t(\varphi) \,,\]
and, by \eqref{EqTr},
\[ \varphi = -t^{k-2m}U_t(\varphi^{Fr}) \,.\]
Hence $U_t$ acts on $Ker(Tr)^{Fr}$ as the Fricke involution (modulo the sign) and the eigenvalues should be the same. Indeed
\begin{equation}\label{EqNewEig}
\lambda^2\varphi = U_t^2(\varphi) = -U_t(\varphi^{Fr}) = t^{2m-k}\varphi\,,
\end{equation}
the eigenvalues on newforms are $\pm\sqrt{t^{2m-k}}$, i.e., their slope is $m-\frac{k}{2}$.

Note that if we assume that something of the form $\delta(\varphi,-\frac{1}{\lambda}\varphi)$, for some $\T_t$-eigenform
$\varphi$ as above, is new (i.e., of type $\psi^{Fr}$ for some $\psi\in Ker(Tr)$), we obtain (by \eqref{EqFricke} and
\eqref{Eqgammatdelta})
\begin{align*}
t^{2m-k}\psi & = (\psi^{Fr})^{Fr} = \delta(\varphi,-\frac{1}{\lambda}\varphi)^{Fr} \nonumber \\
\ & = \delta_t\varphi-\frac{1}{\lambda}t^{2m-k}\delta_1\varphi = \delta(-\frac{1}{\lambda}t^{2m-k}\varphi,\varphi)\,,
\end{align*}
i.e.,
\begin{equation}\label{Eq?}
\psi = \delta(-\frac{1}{\lambda}\varphi,t^{k-2m}\varphi)\,.
\end{equation}
Now $\psi\in Ker(Tr)$ and equations \eqref{Eq1Utdelta} and \eqref{Eq2Utdelta} yield
\begin{align}
\psi & = -t^{k-2m} U_t(\psi^{Fr}) = -t^{k-2m} U_t(\delta(\varphi,-\frac{1}{\lambda}\varphi)) \nonumber \\
\ & = -t^{k-2m}\delta(\T_t\varphi,-\varphi) = -t^{k-2m}\delta(\lambda\varphi,-\varphi)\ .
\end{align}
Comparing this with \eqref{Eq?} and using the injectivity of $\delta$ one finds
\[ -\frac{1}{\lambda}=-t^{k-2m}\lambda \quad{\rm ,\ i.e.,}\quad \lambda^2=t^{2m-k} \]
as predicted by the slope for newforms. We could not find any contradiction here to the fact of
having a form new and old at the same time, anyway our computations never found the slope $m-\frac{k}{2}$
among the ones for $\T_t$ and we think that there should be a trivial intersection (as in the classical case as a consequence of
the Ramanujan conjecture)
between oldforms and newforms.

\subsection{Slopes, inseparability and non-diagonalizability}\label{SecSlopInsep}
The slope $m-\frac{k}{2}$ hints at the presence of inseparable eigenvalues whenever $k$ is odd and the characteristic is even
(the only case in which there are nonzero cuspidal forms of odd weight). We briefly recall that inseparable eigenvalues
immediately imply the non-diagonalizability of our matrices. Indeed let $K$ be any field and denote by
$\mathcal{M}_n(K)$ the $n\times n$ matrices with coefficients in $K$.

\begin{defin}
The {\em minimal polynomial} of $A\in\mathcal{M}_n(K)$ is the non-zero monic polynomial $f\in K [X]$ with least
degree such that $f(A)=0$.
\end{defin}

The minimal and characteristic polynomial of $A\in\mathcal{M}_n(K)$ have the same irreducible factors in $K[X]$
(see \cite[Corollary 4.10]{Co}) and $A$ is diagonalizable over $K$ if and only if its minimal polynomial in $K[X]$
splits in $K[X]$ and has distinct roots (\cite[Theorem 4.11]{Co}).\\
We have the following

\begin{thm}\label{ThmMinPol}
Let $V$ be a finite dimensional $\C_\infty$-vector space and $T:V\to V$ be a linear operator defined over $F$.
If the characteristic polynomial of $T$ has a root $\alpha$ which is inseparable over $F$, then $T$ is not
diagonalizable over $\C_\infty$.
\end{thm}

\begin{proof}
Obviously the minimal polynomial of $A$ belongs to $F[X]$, therefore it is divisible by the minimal polynomial
of $\alpha$ over $F$. It now suffice to note that the minimal polynomial does not depend on the base field, hence
it does not have distinct roots in $\C_\infty$.
\end{proof}

\noindent Therefore $U_t$ has an inseparable eigenvalue (and is not diagonalizable) whenever $q$ is even, $k$ is odd and
$Ker(Tr)\cap Ker(Tr')\neq \{0\}$ (i.e., the space of newforms is nontrivial).
Let $\dim(S_{k,m}^1(\G_0(1)))=\alpha$ and $\dim(S_{k,m}^1(\G_0(t)))=\beta$, because of the injectivity of $\delta$ one
obviously has $\beta\geqslant 2\alpha$. Since $Tr$ and $Tr'$ are linear maps from $S_{k,m}^1(\G_0(t))$ to
$S_{k,m}^1(\G_0(1))$ one has $\dim Ker(Tr),\dim Ker(Tr') \geqslant \beta-\alpha$, moreover
$Ker(Tr)+Ker(Tr')\subset S_{k,m}^1(\G_0(t))$ so it has dimension $\leqslant \beta$. By Grassmann formula we have
\[ \dim (Ker(Tr)\cap Ker(Tr')) =\dim Ker(Tr)+\dim Ker(Tr')-\dim(Ker(Tr)+Ker(Tr'))
\geqslant \beta-2\alpha\,. \]
Therefore $\beta >2\alpha$ yields $Ker(Tr)\cap Ker(Tr')\neq \{0\}$ and $U_t$ is not diagonalizable when $k$ is odd and $q$ is even.

In the remaining case $\beta = 2\alpha$ with $Ker(Tr)\cap Ker(Tr')= \{0\}$, we can actually prove non diagonalizability in general.

\begin{thm}\label{ThmNonDiag2}
If $\beta=2\alpha$ and $Ker(Tr)\cap Ker(Tr')=\{0\}$, then $U_t$ is not diagonalizable (for any $k$ and any $q$).
\end{thm}

\begin{proof}
The hypotheses immediately imply:\begin{itemize}
\item by Grassmann
\[ \beta \geqslant \dim(Ker(Tr)+Ker(Tr')) = \dim Ker(Tr)+\dim Ker(Tr') \geqslant 2\beta -2\alpha =\beta\, \]
i.e.,
\[ S_{k,m}^1(\G_0(t))= Ker(Tr)\oplus Ker(Tr')\,;\]
\item $\dim S_{k,m}^1(\G_0(t))=2\dim S_{k,m}^1(\G_0(1)) =\dim Im(\delta)$, so $S_{k,m}^1(\G_0(t))=Im (\delta)$,
i.e., $\delta$ is an isomorphism and the whole space $S_{k,m}^1(\G_0(t))$ is generated by {\em oldforms}.
\end{itemize}
Let $\delta_1\varphi+\delta_t\psi$ be an eigenvector of eigenvalue $\lambda\neq 0$, then
\[ U_t(\delta_1\varphi+\delta_t\psi)=\delta_1(\T_t\varphi)-\delta_t\varphi=\lambda(\delta_1\varphi+\delta_t\psi)\,. \]
The injectivity of $\delta$ yields
\[ \T_t\varphi=\lambda\varphi\quad{\rm and}\quad -\varphi=\lambda\psi\quad {\rm (hence\ } \T_t\psi=\lambda\psi{\rm)} \]
so that
\[ \delta_1\varphi+\delta_t\psi = -\lambda\delta_1\psi+\delta_t\psi\,.\]
Thus
\begin{align*} Tr(\delta_1\varphi+\delta_t\psi) & = Tr(-\lambda\delta_1\psi+\delta_t\psi) \\
\ & = -\lambda\psi + \T_t\psi =-\lambda\psi+\lambda\psi=0
\end{align*}
(this provides a certain analogue of Section \ref{SecSlopesOld}: from an old eigenvector of eigenvalue
$\lambda$ for $U_t$ we derived the existence of an eigenvector of eigenvalue $\lambda$ for $\T_t$).

\noindent Therefore all eigenvectors of eigenvalue different from 0 are inside $Ker(Tr)$ and, to be able to find a basis
of eigenvectors for $S_{k,m}^1(\G_0(t))=Ker(Tr)\oplus Ker(Tr')$, we need $Ker(Tr')$ to be contained in the eigenspace
of eigenvalue 0, i.e., in $Ker(U_t)$. But on the subspace $Ker(Tr')$, the operator $U_t$ coincides with the Fricke involution,
hence it has to have some nonzero eigenvalues (unless $Ker(Tr')=0$, but then $S_{k,m}^1(\G_0(t))=0$ and there is nothing to prove).
\end{proof}

\begin{rem}\label{RemNonDiag2}
Together with Section \ref{SecSlopesNew} this proves that if $k$ is odd and $q$ is even, then $U_t$ is not diagonalizable.
Theorem \ref{ThmNonDiag2} holds in general (no inseparable eigenvalues appear there): it would be interesting
to investigate the possibility of having $\beta=2\alpha$ and $Ker(Tr)\cap Ker(Tr')=\{0\}$ for general $k$ and $q$ and we
would be quite surprised if that can actually happen as $k$ grows.
\end{rem}

\section{Action of $U_t$ on cusp forms for $\G_1(t)$} \label{SecGamma1}
The tree $\G_1(t)\backslash \Tr$ has two cusps (\cite[Proposition 5.6]{GN}) corresponding to $[1:0]$ and $[0:1]$.
The path connecting the two cusps is a fundamental domain for $\G_1(t)$. Here is a picture of it
\[ \xymatrix { .. v_{-2,0}=\matrix {1}{0}{0}{t^2} \ar@/^2.0pc/@{<-}[r]^{\overline{e}_{-2,0}=\matrix {0}{1}{t}{0}}
\ar@/_2.0pc/[r]_{e_{-2,0}=\matrix {1}{0}{0}{t^2}} &
v_{-1,0}=\matrix {1}{0}{0}{t} \ar@/^2.0pc/@{<-}[r]^{\overline{e}_{-1,0}=\matrix {0}{1}{1}{0}}
\ar@/_2.0pc/[r]_{e_{-1,0}=\matrix {1}{0}{0}{t}} &
v_{0,0}=\matrix {1}{0}{0}{1} \ar@/^2.0pc/@{<-}[r]^{\overline{e}_{0,0}=\matrix {0}{t}{1}{0}}
\ar@/_2.0pc/[r]_{e_{0,0}=\matrix {1}{0}{0}{1}} &
v_{1,0}=\matrix {t}{0}{0}{1} \ar@/^2.0pc/@{<-}[r]^{\overline{e}_{1,0}=\matrix {0}{t^2}{1}{0}}
\ar@/_2.0pc/[r]_{e_{1,0}=\matrix {t}{0}{0}{1}} &
v_{2,0}..
 } \]
This fundamental domain does not contain stable vertices but one stable edge, namely $\e:=\overline{e}_{-1,0}$.

\subsection{The action of $U_t$}
We compute the action of $U_t$ at $\c(\e)$ for any harmonic cocycle $\c$. \medskip

\noindent \underline{{\bf Warning.}} {\em To shorten some notations and simplify some exponents we normalize Hecke operators
multiplying by the factor $t^{k-m}$ (as in \cite{LM}): this will erase any reference to the type $m$ in the final formulas.
Moreover the slope for newforms will transform into $\frac{k}{2}$ (from the previous $m-\frac{k}{2}$). We shall recall
(and reverse) this normalization in some explicit computations of eigenforms in Section \ref{SecAntidiagonalNewforms}
and Example \ref{Ex359}.}

\begin{align*}
& U_t(\c(\e))(X^jY^{k-2-j}) =
t^{k-m} \bigg\{ \sum_{\b\in\F_q} \matrix {1}{\b}{0}{t}^{-1}\c\left( \matrix {1}{\b}{0}{t}\e \right)\bigg\}(X^jY^{k-2-j})  \\
& =t^{k-m}\left\{\matrix {1}{0}{0}{t}^{-1}\c\left( \matrix {1}{0}{0}{t} \e \right) +
\sum_{\b\in\F_q^*} \matrix {1}{\b}{0}{t}^{-1}\c\left( \matrix {1}{\b}{0}{t}\e \right)\right\}(X^jY^{k-2-j})\, .
\end{align*}
Split the above equation in two parts
\begin{gather*}
A:= t^{k-m} \left\{\matrix {1}{0}{0}{t}^{-1}\c\left( \matrix {1}{0}{0}{t} \e \right)\right\} (X^jY^{k-2-j})\, ,\\
B:= t^{k-m} \left\{ \sum_{\b\in\F_q^*} \matrix {1}{\b}{0}{t}^{-1}\c\left( \matrix {1}{\b}{0}{t}\e \right)\right\}(X^jY^{k-2-j})\, .
\end{gather*}

Let us start with $A$. Then
\begin{align*}
A & = t^{k-m}\matrix {1}{0}{0}{\frac{1}{t}} \c\left(\matrix{1}{0}{0}{t}\matrix{0}{1}{1}{0}\right)(X^jY^{k-2-j}) \\
& = t^{k-m}\left(\frac{1}{t}\right)^{1-m} \c\left(\matrix{0}{1}{t}{0}\right)\left(X^j\left(\frac{Y}{t}\right)^{k-2-j}\right) \\
& = t^{j+1}\c\left(\matrix{0}{1}{t}{0}\right)(X^jY^{k-2-j}) \,.
\end{align*}
Now use harmonicity
\begin{equation}\label{EqHar1}
\c\left(\matrix{0}{1}{t}{0}\right) = -\sum_{v\in \mathbb{F}_q} \c\left(\matrix{1}{0}{vt}{1} \matrix{1}{0}{0}{t}\right)
\end{equation}
(with $\bigl(\begin{smallmatrix}1&0\\vt&1\end{smallmatrix}\bigr)\in \Gamma_1(t)$) to get
\begin{align*}
A & = -t^{j+1}\left\{\c\left(\matrix{1}{0}{0}{1} \matrix{1}{0}{0}{t}\right) +
\sum_{v\in \mathbb{F}_q^*} \c\left(\matrix{1}{0}{vt}{1} \matrix{1}{0}{0}{t}\right)\right\}(X^jY^{k-2-j}) \\
& = -t^{j+1}\left\{\c\left(\matrix{1}{0}{0}{t}\right) +
\sum_{v\in \mathbb{F}_q^*} \matrix{1}{0}{vt}{1}\c\left(\matrix{1}{0}{0}{t}\right)\right\}(X^jY^{k-2-j}) \\
& = - t^{j+1}\c\left(\matrix{1}{0}{0}{t}\right)(X^jY^{k-2-j}) -
t^{j+1}\sum_{v\in \mathbb{F}_q^*} \c\left(\matrix{1}{0}{0}{t}\right)\left(X^j(vtX+Y)^{k-2-j}\right) \\
& = -t^{j+1}\c(e)(X^jY^{k-2-j})-
t^{j+1}\sum_{v\in \mathbb{F}_q^*} \c(e)\left(\sum_{n=0}^{k-2-j} \binom{k-2-j}{n} v^nt^nX^{j+n}Y^{k-2-j-n}\right) \\
& = -t^{j+1}\c(e)(X^jY^{k-2-j}) - t^{j+1} \c(e)
\left(\sum_{n=0}^{k-2-j} \binom{k-2-j}{n} t^nX^{j+n}Y^{k-2-j-n} \sum_{v\in \mathbb{F}_q^*} v^n\right) \\
& = -t^{j+1}\c(e)(X^jY^{k-2-j}) + t^{j+1} \c(e)
\left(\sum_{\begin{subarray}{c} n\equiv 0\pmod{q-1} \\ n\geqslant 0\end{subarray}} \binom{k-2-j}{n} t^nX^{j+n}Y^{k-2-j-n} \right) \\
& = t^{j+1}\c(e)
\left(\sum_{\begin{subarray}{c} n\equiv 0\pmod{q-1} \\ n\geqslant q-1\end{subarray}} \binom{k-2-j}{n} t^nX^{j+n}Y^{k-2-j-n} \right)
\end{align*}
(the first term of the sum cancels with the term on the left, sign has changed because of
$\displaystyle{\sum_{v\in \mathbb{F}_q^*} v^n =-1}$ for $n\equiv 0\pmod{q-1}$).

As for part $B$, note that
\begin{align*}
\matrix{1}{\b}{0}{t}\e & = \matrix{1}{\b}{0}{t} \matrix{0}{1}{1}{0} = \matrix{\b}{1}{t}{0}\\
& = \matrix{1}{0}{\frac{t}{\b}}{1}\matrix{1}{0}{0}{t}\matrix{\b}{1}{0}{-\frac{1}{\b}} =: \g_\b\, e\, \iota_\b
\end{align*}
where $\g_\b\in \Gamma_1(t)$ and $\iota_\b\in \I(F_\infty)$. Then
\begin{align*}
B & = t^{k-m}
\sum_{\b\in\mathbb{F}_q^*} \matrix{1}{-\frac{\b}{t}}{0}{\frac{1}{t}} \c(\g_\b\, e\, \iota_\b)(X^jY^{k-2-j}) \\
& = t^{k-m} \sum_{\b\in\mathbb{F}_q^*}
\left(\matrix{1}{-\frac{\b}{t}}{0}{\frac{1}{t}} \g_\b\right) \c(e )(X^jY^{k-2-j}) \\
& = t^{k-m} \sum_{\b\in\mathbb{F}_q^*} \matrix{0}{-\frac{\b}{t}}{\frac{1}{\b}}{\frac{1}{t}} \c(e)(X^jY^{k-2-j}) \\
& = t^{k-m} \left(\frac{1}{t}\right)^{1-m} \c(e )
\left( \sum_{\b\in\mathbb{F}_q^*} \left(-\frac{\b}{t}Y\right)^j\left( \frac{X}{\b}+\frac{Y}{t}\right)^{k-2-j}\right) \\
& = t \c(e) \left(\sum_{\b\in\mathbb{F}_q^*} (-1)^j \b^{2+2j-k} Y^j(tX+\b Y)^{k-2-j} \right)\\
& = t \c(e )
\left(\sum_{\b\in\mathbb{F}_q^*} \sum_{n=0}^{k-2-j} (-1)^j \b^{2+2j-k} Y^j \binom{k-2-j}{n} t^nX^n \b^{k-2-j-n}Y^{k-2-j-n} \right) \\
& = t \c(e )\left( \sum_{n=0}^{k-2-j} \binom{k-2-j}{n} (-1)^j t^nX^n Y^{k-2-n} \sum_{\b\in\mathbb{F}_q^*} \b^{j-n}\right) \\
& = -t \c(e )
\left(\sum_{\begin{subarray}{c} n\equiv j\pmod{q-1} \\ n\geqslant 0\end{subarray}} \binom{k-2-j}{n} (-1)^j t^nX^n Y^{k-2-n} \right) \,.
\end{align*}
To make it more compatible with the term $A$ (obtained for $\b=0$) we change the index using $n=\ell+j$, to get
\begin{align*}
B & =  -t \c(e )
\left(\sum_{\begin{subarray}{c} \ell\equiv 0\pmod{q-1} \\ \ell\geqslant -j\end{subarray}}
\binom{k-2-j}{\ell+j} (-1)^j t^{\ell+j}X^{\ell+j} Y^{k-2-j-\ell} \right) \\
& = (-t)^{j+1} \c(e )
\left(\sum_{\begin{subarray}{c} \ell\equiv 0\pmod{q-1} \\ \ell\geqslant -j\end{subarray}}
\binom{k-2-j}{\ell+j} t^\ell X^{\ell+j} Y^{k-2-j-\ell} \right) \,.
\end{align*}

Putting all together, the final general formula is
\begin{align}\label{FinEqTt}
 U_t(\c(\e))(X^jY^{k-2-j}) & =
t^{j+1}\c(e) \left\{ \sum_{\begin{subarray}{c} n\equiv 0\pmod{q-1} \\
n\geqslant q-1\end{subarray}} \binom{k-2-j}{n} t^nX^{j+n}Y^{k-2-j-n} + \right.  \\
 & \qquad +\left.  (-1)^{j+1} \sum_{\begin{subarray}{c} n\equiv 0\pmod{q-1} \\
n\geqslant -j\end{subarray}} \binom{k-2-j}{n+j} t^nX^{j+n} Y^{k-2-j-n} \right\}\,. \nonumber
\end{align}

\subsection{Action of $U_t$ on a particular basis} \label{UBasis}
To check the diagonalizability of the operator $U_t$, we fix the same basis used
in \cite{LM}.

Recall that by \cite[Corollary 5.7]{GN} we have $g(\G_1(t))=0$ and, applying Theorem \ref{DimTorFree}, we obtain
\begin{itemize}
\item[{\bf 1.}] {$\dim_{\C_\infty} S^1_k(\G_1(t))=k-1$;}
\item[{\bf 2.}] {$\dim_{\C_\infty} S^2_k(\G_1(t))= \left\{ \begin{array}{cr}
0 & k=2 \\
\ & \\
k-3 & k>2
\end{array}\right.$\ .}
\end{itemize}

For any $j\in\{0,1,\dots,k-2\}$, let $\c_j(\e)$ be defined by
\[ \c_j(\e)(X^iY^{k-2-i})=\left\{ \begin{array}{ll} 1 & {\rm if}\ i=j \\
\ & \\
0 & {\rm otherwise} \end{array} \right. \ .\]
The sets $\mathcal{B}^1_k(\G_1(t)):=\{\c_j(\e)\,,\,0\leqslant j\leqslant k-2\}$ and
$\mathcal{B}^2_k(\G_1(t)):=\{\c_j(\e)\,,\,1\leqslant j\leqslant k-3\}$ are bases for
$S^1_k(\G_1(t))$ and $S^2_k(\G_1(t))$ respectively (recall that the elements of $S^2_k(\G_1(t))$
have to vanish on $X^{k-2}$ and $Y^{k-2}$). We shall work mainly on $\mathcal{B}^1_k(\G_1(t))$, the results for
$\mathcal{B}^2_k(\G_1(t))$ will easily follow and we shall point them out in some remarks along the way.
Then, unless otherwise stated, $U_t$ is understood to act on $S^1_k(\G_1(t))$.

Specializing equation \eqref{FinEqTt} at $\c=\c_j$, we get nonzero values only for $n=0$, i.e.,
the first sum vanishes and the second one is reduced to
\begin{align*}
 U_t(\c_j(\e)) (X^jY^{k-2-j}) & = t^{j+1}\c_j(e) \left( (-1)^{j+1} \binom{k-2-j}{j} X^j Y^{k-2-j} \right) \\
& = -(-t)^{j+1} \binom{k-2-j}{j} \c_j(\e) (X^j Y^{k-2-j} ) = -(-t)^{j+1} \binom{k-2-j}{j} \ .
\end{align*}
Now we look for the value of $U_t(\c_j)(\e)$ at $X^iY^{k-2-i}$ for any $i\neq j$, so that equation \eqref{FinEqTt} becomes
\begin{equation}\label{Eqineqj}
\begin{split}
U_t(\c_j(\e)) (X^iY^{k-2-i})& = t^{i+1}\c_j(e) \left\{ \sum_{\begin{subarray}{c} n\equiv 0\pmod{q-1} \\
n\geqslant q-1\end{subarray}} \binom{k-2-i}{n} t^nX^{i+n}Y^{k-2-i-n}+ \right. \\
& \left. \quad + (-1)^{i+1}
\sum_{\begin{subarray}{c} n\equiv 0\pmod{q-1} \\ n\geqslant -i\end{subarray}} \binom{k-2-i}{n+i} t^nX^{i+n} Y^{k-2-i-n} \right\} \ .
\end{split}
\end{equation}
To get nonzero values one needs $n+i=j$ and $n\equiv 0\pmod{q-1}$, so we write $i=j+h(q-1)$ for some $h\in \Z-\{0\}$
(note also that in any case $(-1)^{i+1}=(-1)^{j+1}$) and get
\begin{align*}
 U_t(\c_j(\e)) & (X^iY^{k-2-i}) \\
& = -t^{i+1}\c_j(\e) \left( \binom{k-2-j-h(q-1)}{-h(q-1)} t^{-h(q-1)}X^jY^{k-2-j} \right.\\
& \left. \qquad  + (-1)^{i+1} \binom{k-2-j-h(q-1)}{j} t^{-h(q-1)}X^j Y^{k-2-j} \right) \\
& = -t^{j+1}\left[ \binom{k-2-j-h(q-1)}{-h(q-1)} + (-1)^{j+1} \binom{k-2-j-h(q-1)}{j} \right] \c_i(\e) (X^iY^{k-2-i}).
\end{align*}
Note that many coefficients would be 0, for example $\binom{k-2-j-h(q-1)}{-h(q-1)}$ vanishes for any $h>0$ and
$\binom{k-2-j-h(q-1)}{j}$ vanishes whenever $k-2-j-h(q-1)<0$, but this is useful for a general formula.\bigskip

Finally we have
\begin{align}\label{Ttcj}
U_t(\c_j(\e)) & = -(-t)^{j+1} \binom{k-2-j}{j} \c_j(\e) -t^{j+1}\sum_{h\neq 0}\left[ \binom{k-2-j-h(q-1)}{-h(q-1)} \right.\\
\ &\left. + (-1)^{j+1} \binom{k-2-j-h(q-1)}{j} \right] \c_{j+h(q-1)}(\e)  \nonumber
\end{align}
(where it is understood that $\c_{j+h(q-1)}(\e) \equiv 0$ whenever $j+h(q-1)<0$ or $j+h(q-1)>k-2$).

From formula \eqref{Ttcj} one immediately notes that the $\c_j$ can be divided into classes modulo $q-1$ and
every such class is stable under the action of $U_t$.
We shall denote by $C_j$ the class of $\c_j(\e)$, i.e, $C_j=\{\c_j(\e),\c_{j+(q-1)}(\e),\dots\}$: the cardinality
of $C_j$ is the largest integer $n$ such that $j+(n-1)(q-1)\leqslant k-2$ (note that it is possible to have
$|C_j|=0$, exactly when $j>k-2$). Reordering the basis in the following way
$\mathcal{B}^1_k(\G_1(t))=\{C_0,C_1,\dots,C_{q-2}\}$, one has that the matrix associated to the action of $U_t$ has (at most) $q-1$
blocks (of dimensions $|C_j|$, $0\leqslant j\leqslant q-2$) on the diagonal and 0 everywhere else. Obviously
the matrix is diagonalizable if and only if each block is. In particular $|C_j|\leqslant 1$ (i.e., $j+(q-1)>k-2$)
always yields a diagonal (or empty) block.

Fix $j$ and $k$ and assume that $C_j$ has exactly $n$ elements (i.e., the last one is $\c_{j+(n-1)(q-1)}(\e)$). Let
$m_{a,b}(j,k)$ be the coefficient in the $a$-th row and $b$-th column of the matrix associated to the block $C_j$, then
$m_{a,b}(j,k)$ is the coefficient of $\c_{j+(a-1)(q-1)}(\e)$ in $U_t(\c_{j+(b-1)(q-1)}(\e))$.
Using formula \eqref{Ttcj} one finds
\begin{equation}\label{EqCoeffMj}
m_{a,b}(j,k) = \left\{ \begin{array}{ll} \displaystyle{-t^{j+1+(b-1)(q-1)}\left[\binom{k-2-j-(a-1)(q-1)}{(b-a)(q-1)} \right.} & \\
\displaystyle{\left. +(-1)^{j+1+(b-1)(q-1)}\binom{k-2-j-(a-1)(q-1)}{j+(b-1)(q-1)}\right]} & {\rm if}\ a\neq b \\
\ & \\
\displaystyle{-(-t)^{j+1+(a-1)(q-1)}\binom{k-2-j-(a-1)(q-1)}{j+(a-1)(q-1)}} & {\rm if}\ a = b \end{array}\right.
\end{equation}
(for future reference note that for any $q$ one has $(-1)^{(\ell-1)(q-1)}=1$ for any $\ell$).

\subsection{The blocks associated to $S^1_{k,m}(\G_0(t))$}\label{SecSpBlocks}
A set of representatives for $\G_0(t)/\G_1(t)$ is provided by the $(q-1)^2$ matrices
$R^0_1=\left\{ \matrix{a}{0}{0}{d}\,:\,a,d\in \F_q^*\right\}$, hence a cocycle $\c_j$ comes from $S^1_{k,m}(\G_0(t))$
if and only if it is $R^0_1$-invariant.

\noindent A computation similar to the previous ones leads to
\[ \matrix{a}{0}{0}{d}^{-1}\c_j\left(\matrix{a}{0}{0}{d}\e\right)(X^\ell Y^{k-2-\ell})=
a^{m-1-\ell}d^{m-k+\ell+1}\c_j(\e)(X^\ell Y^{k-2-\ell})\,.\]
Therefore
\[ \matrix{a}{0}{0}{d}\cdot\c_j=a^{m-1-j}d^{m-k+j+1}\c_j\quad \forall j \]
and this is $R^0_1$-invariant if and only if
\[ a^{m-1-j}d^{m-k+j+1}=1 \quad \forall a,d\in \F_q^*\,.\]
This yields
\[ j\equiv m-1\equiv k-m-1 \pmod{q-1} \,,\]
i.e.,
\[ k\equiv 2j+2 \pmod{q-1} \]
(and $k\equiv 2m\pmod{q-1}$ as natural to get a nonzero space of cuspidal forms for $\G_0(t)\,$). If $q$ is even this
provides a unique class $C_j$, if $q$ is odd then we have two solutions $j$ and $j+\frac{q-1}{2}$. Note that in any case
$k$ has the form $2j+2+(n-1)(q-1)$ for some integer $n$ and the classes corresponding to $S^1_{k,m}(\G_0(t))$
are determined by the type $m$ (as predictable since $m$ plays a role in $S^1_{k,m}(\G_0(t))$ but not in $S^1_k(\G_1(t))$,
because all matrices in $\G_1(t)$ have determinant 1).

\subsection{Dimension of $S^1_{k,m}(\G_0(t))$-blocks}
Assume $q$ is even, take $j$ as the unique solution of $k\equiv 2j+2\pmod{q-1}$ and write $k=2j+2+(n-1)(q-1)$
(note that $k$ is even if and only if $n$ is odd).
The class $C_j$ has exactly $n$ elements. Indeed $j+(n-1)(q-1)$ is obviously
$\leqslant k-2=2j+(n-1)(q-1)$ and
\[ j+n(q-1)=k-2-j+(q-1)\geqslant k-2-(q-2)+(q-1) = k-1\,.\]

When $q$ is odd there are two solutions to $k\equiv 2j+2\pmod{q-1}$ and we put $j$ for the smallest (positive) one
and $j+\frac{q-1}{2}$ for the other (note that the classes $C_j$ and $C_{j+\frac{q-1}{2}}$ are always disjoint).
As before we can write
\[ k=2j+2+(n-1)(q-1)=2(j+\frac{q-1}{2})+2+(n-2)(q-1) \]
and it is easy to check that
\[ |C_j|=n\quad{\rm and}\quad |C_{j+\frac{q-1}{2}}|=n-1 \,.\]

\begin{rem}\label{RemDimCusp}
This could be seen as an easy alternative to the Rieman-Roch argument usually used to
compute the dimension of such spaces, see for example \cite[Section 4]{Cor}.
\end{rem}

\subsection{Matrices for the block $C_j$}\label{SecSymmetry}
We denote by $M(j,n,q,t)$ the matrix associated to the action of $U_t$ on $C_j$ (for $k=2j+2+(n-1)(q-1)\,$)
and by $M(j,n,q)$ the corresponding coefficient matrix (i.e., the one without the powers of $t$).
To avoid repetitions we shall consider the block associated to $C_j$ of dimension $n$ with $k=2j+2+(n-1)(q-1)$,
(formulas for $C_{j+\frac{q-1}{2}}$ are the same, just substitute $j$ with $j+\frac{q-1}{2}$ and take into account
the different parity of the dimension $n-1$).

Using formula \eqref{EqCoeffMj} we see that the general entries of $M(j,n,q)$ are
\begin{equation}\label{spam}
\displaystyle{ m_{a,b}= \left\{\begin{array}{ll}
\displaystyle{-\left[\binom{j+(n-a)(q-1)}{j+(n-b)(q-1)} + (-1)^{j+1} \binom{j+(n-a)(q-1)}{j+(b-1)(q-1)}\right]} & {\rm if}\ a\neq b \\
\ & \\
\displaystyle{(-1)^{j+2}\binom{j+(n-a)(q-1)}{j+(a-1)(q-1)}} & {\rm if}\ a=b \end{array}\right. \,.}
\end{equation}
It is easy to check that $M(j,n,q)$ satisfies some symmetry relations. In particular
\begin{itemize}\label{symmetry}
\item[1.] {\em symmetry between columns}: $m_{a,n+1-b}=(-1)^{j+1}m_{a,b}$ for any $a\neq b, n+1-b$
(i.e., outside diagonal and antidiagonal);
\item[2.] {\em symmetry between diagonal and antidiagonal}: $m_{a,n+1-a}=(-1)^{j+1}(m_{a,a}-1)$ for any $a\neq n+1-a$;
\item[3.] {\em central column for odd $n$}: $m_{\frac{n+1}{2},\frac{n+1}{2}}=(-1)^{j+2}$ and
\[ m_{a,\frac{n+1}{2}}=-\binom{j+(n-a)(q-1)}{j+\frac{n-1}{2}(q-1)}(1+(-1)^{j+1})\quad {\rm for}\ a\neq \frac{n+1}{2}\,,\]
which is 0 for $a\geqslant \frac{n+3}{2}$.
\end{itemize}

Moreover for the elements on the antidiagonal (i.e., the ones with $b=n+1-a$), we only need to check the first $\frac{n}{2}$
(resp. $\frac{n-1}{2}$) columns if $n$ is even (resp. if $n$ is odd) thanks to the symmetries above. One has\begin{itemize}
\item[4.] {\em antidiagonal, $n$ even and $\frac{n}{2}+1\leqslant a\leqslant n$}:
\[ m_{a,n+1-a}=-\left[\binom{j+(n-a)(q-1)}{j+(a-1)(q-1)} + (-1)^{j+1} \binom{j+(n-a)(q-1)}{j+(n-a)(q-1)}\right]=(-1)^{j+2} \]
(because in our range $n-a<a-1$);
\item[5.] {\em antidiagonal, $n$ odd and $\frac{n+3}{2}\leqslant a\leqslant n$}:
\[ m_{a,n+1-a}=-\left[\binom{j+(n-a)(q-1)}{j+(a-1)(q-1)} + (-1)^{j+1} \binom{j+(n-a)(q-1)}{j+(n-a)(q-1)}\right]=(-1)^{j+2} \]
as well (because in our range $n-a<a-1$ again).
\end{itemize}
Note that 4 and 5, together with 2, yield
\[ (-1)^{j+2}=(-1)^{j+1}(m_{a,a}-1)\,,\quad{\rm i.e.,}\quad m_{a,a}=0 \]
in the range in which 4 and 5 hold.

Finally below the antidiagonal (i.e., for $b>n+1-a$), checking only the first columns as above, we find
\begin{itemize}
\item[6.] {\em below antidiagonal, $n$ even and $\frac{n}{2}+1\leqslant a\leqslant n-1$}:
\[ -\left[\binom{j+(n-a)(q-1)}{j+(n-b)(q-1)} + (-1)^{j+1} \binom{j+(n-a)(q-1)}{j+(b-1)(q-1)}\right]=0 \]
(because in our range $n-a<n-b$ and $n-a<b-1$);
\item[7.] {\em below antidiagonal, $n$ odd and $\frac{n+3}{2}\leqslant a\leqslant n-1$}:
\[ -\left[\binom{j+(n-a)(q-1)}{j+(n-b)(q-1)} + (-1)^{j+1} \binom{j+(n-a)(q-1)}{j+(b-1)(q-1)}\right]=0 \]
as well (because in our range $n-a<n-b$ and $n-a<b-1$ again).
\end{itemize}

Putting all these information together we can see that for any even $n$ the matrix $M(j,n,q)$
has the following shape
\[ \left(\begin{array}{cccccccc} m_{1,1} & m_{1,2} & \cdots & m_{1,\frac{n}{2}} & (-1)^{j+1}m_{1,\frac{n}{2}}
& \cdots & (-1)^{j+1}m_{1,2} & (-1)^{j+1}(m_{1,1}-1)\\
m_{2,1} & m_{2,2} & \cdots & m_{2,\frac{n}{2}} & (-1)^{j+1}m_{2,\frac{n}{2}} & \cdots & (-1)^{j+1}(m_{2,2}-1) & (-1)^{j+1}m_{2,1}\\
\vdots & \vdots &   & \vdots & \vdots &  & \vdots & \vdots\\
m_{\frac{n}{2},1} & m_{\frac{n}{2},2} & \cdots & m_{\frac{n}{2},\frac{n}{2}} & (-1)^{j+1}(m_{\frac{n}{2},\frac{n}{2}}-1) & \cdots &
(-1)^{j+1}m_{\frac{n}{2},2} & (-1)^{j+1}m_{\frac{n}{2},1}\\
m_{\frac{n}{2}+1,1} & m_{\frac{n}{2}+1,2} & \cdots & (-1)^{j+2} & 0 & \cdots & (-1)^{j+1}m_{\frac{n}{2}+1,2}
& (-1)^{j+1}m_{\frac{n}{2}+1,1}\\
\vdots & \vdots & \udots  & \vdots & \vdots & \ddots & \vdots & \vdots\\
m_{n-1,1} & (-1)^{j+2} & \cdots  & 0 & 0 &  \cdots & 0 & (-1)^{j+1}m_{n-1,1}\\
(-1)^{j+2} & 0 & \cdots  & 0 & 0 & \cdots & 0 & 0
\end{array}
\right) \]
while, for odd $n$, one simply needs to modify the indices a bit and add the central $\frac{n+1}{2}$-th column
\[ \begin{array}{c} m_{1,\frac{n+1}{2}} \\ \vdots \\ m_{\frac{n-1}{2},\frac{n+1}{2}} \\ (-1)^{j+2} \\ 0 \\ \vdots \\ 0 \end{array}\,.\]

\section{Diagonalizability of $U_t$ on blocks associated to $S^1_{k,m}(\G_0(t))$} \label{Char2Sec}
We present here some results which provide some evidence for the diagonalizability of $U_t$.
Since blocks of dimension 1 are already diagonal, we shall tacitly assume $|C_j|\geqslant 2$ ignoring
the trivial cases $|C_j|=0,1$.
All theorems will rely on direct computation of binomial coefficients and the bounds
required for some of the parameters are actually sharp as we shall show with some examples.
Most of the computations will rely on the following well known

\begin{lem}\label{KummerThm}(Lucas's Theorem)
Let $n,m\in\N$ with $m\leqslant n$ and write their $p$-adic expansions as
$n=n_0+n_1p+\dots +n_d p^d$, $m=m_0+m_1 p + \dots +m_d p^d$. Then
\[ \binom{n}{m} \equiv \binom{n_0}{m_0} \binom{n_1}{m_1} \dots \binom{n_d}{m_d} \pmod p\,.\]
\end{lem}

\begin{proof}
See \cite{DW} or \cite{Gr}.
\end{proof}

\subsection{Diagonalizability of $M(j,n,q)$: case $n\leqslant j+1$}
We can now prove the following

\begin{thm}\label{ThmAntidiagonal}
Let $n\in\N$ and $0\leqslant j\leqslant q-2$. Then, for all $n\leqslant j+1$, the matrix $M(j,n,q)$ is antidiagonal.
\end{thm}

\begin{proof}
Thanks to the symmetries and the properties of the matrices $M(j,n,q)$ mentioned above we are going to analyze the general
$b$-th column only for $1\leqslant b \leqslant \frac{n}{2}$ (or $\leqslant \frac{n+1}{2}$ according to the parity of $n$)
and above the antidiagonal (i.e., for $b<n+1-a$).\\
Let us start with the elements on the diagonal. We rewrite them as
\[  m_{a,a}=(-1)^{j+2}\binom{(n-a)q +j +a-n}{(a-1)q+j-a+1}\,.  \]
Our hypotheses on $j$ and $n$ yield $0\leqslant j+a-n,j-a+1 < q=p^r$, hence, in order to use Lemma \ref{KummerThm}, we can write the
$p$-adic expansion of the terms in the binomial coefficient as
\begin{align*}
(n-a)q+j+a-n & =\a_0+\a_1 p+\dots+\a_{r-1}p^{r-1}+(n-a)p^r; \\
(a-1)q+j-a+1 & =\delta_0+\delta_1 p+\dots+\delta_{r-1}p^{r-1}+(a-1)p^r.
\end{align*}
If there exists $i$ such that $\a_i<\delta_i$, then $\binom{\a_i}{\delta_i}=0$ and $m_{a,a}$ is zero. Otherwise, if
for any $i$ one has $\a_i\geqslant\delta_i$, then $j+a-n\geqslant j+1-a$, i.e., $a-1\geqslant n-a$. Again we get $m_{a,a}=0$
unless $a=\frac{n+1}{2}$ where we have already seen that $m_{\frac{n+1}{2},\frac{n+1}{2}}=(-1)^{j+2}$.

The other $m_{a,b}$ are
\[ m_{a,b}= -\left[ \binom{(n-a)q+j+a-n}{(n-b)q+j+b-n} + (-1)^{j+1}\binom{(n-a)q+j+a-n}{(b-1)q+j-b+1}\right] \,.\]
As before, $j+a-n,j-b+1,j+b-n < q$ and all of them are non-negative. Thus, the $p$-adic expansions
of the terms involved in the coefficients are
\begin{align*} (n-a)q+j+a-n & =\a_0+\a_1 p+\dots+\a_{r-1}p^{r-1}+(n-a)p^r; \\
(b-1)q+j-b+1&=\b_0+\b_1 p+\dots+\b_{r-1}p^{r-1}+(b-1)p^r; \\
(n-b)q+j+b-n&=\g_0+\g_1 p+\dots+\g_{r-1}p^{r-1}+(n-b)p^r.
\end{align*}
By Lemma \ref{KummerThm} we have that
\begin{itemize}
\item if there exists $i$ such that $\a_i<\g_i$, then the first binomial coefficient is zero;
\item if there exists $i$ such that $\a_i<\b_i$, then the second binomial coefficient is zero.
\end{itemize}
Otherwise, if for any $i$ one has $\a_i\geqslant\g_i$, then $j+a-n\geqslant j+b-n$ and $n-b\geqslant n-a$.
This implies that the first binomial coefficient is zero
(observe that the equality $a=b$ cannot happen here, the entry $m_{a,a}$ has already been treated).\\
For the second coefficient assume that for any $i$ one has $\a_i\geqslant \b_i$, then
$j+a-n\geqslant j-b+1$, i.e., $b\geqslant n+1-a$, a contradiction to our assumption of being above the antidiagonal.
\end{proof}

\begin{cor}
Let $n\in\N$ and $0\leqslant j\leqslant q-2$. Then, for all $2\leqslant n\leqslant j+1$, the matrix
$M(j,n,q,t)$ is diagonalizable if and only if $q$ is odd.
\end{cor}

\begin{proof}
Putting back the powers of $t$ in the antidiagonal matrix $M(j,n,q)$, one gets
\[ M(j,n,q,t)=(-1)^{j+2}\left(\begin{array} {ccccc} 0 & 0 & \cdots & 0 & t^{j+1+(n-1)(q-1)} \\
\vdots & \vdots & \cdots & t^{j+1+(n-2)(q-1)} & 0 \\
\vdots & \vdots & \udots & 0 & \vdots \\
0 & t^{j+1+(q-1)} & \cdots & \vdots & \vdots \\
t^{j+1} & 0 & \cdots & 0 & 0 \end{array} \right) \,.\]
Recalling that $k=2j+2+(n-1)(q-1)$, these matrices have characteristic polynomial
\[ \det(M(j,n,q,t)-X\!\!\cdot{\bf Id}_n)=\left\{ \begin{array}{ll} (X^2-t^k)^{\frac{n}{2}} & {\rm if}\ n\ {\rm is\ even} \\
\ & \\
(X^2-t^k)^{\frac{n-1}{2}}(-X+(-1)^{j+2} t^{\frac{k}{2}}) & {\rm if}\ n\ {\rm is\ odd} \end{array} \right. \,,\]
and they are all diagonalizable unless $q$ is even.
\end{proof}

\begin{rem}
Since $n\geqslant 2$, our hypotheses on $j$ being greater or equal to $n-1$ implies that in the class $C_j$ we will
never find $\c_0$. Moreover $\c_{k-2}=\c_{2j+(n-1)(q-1)}\in C_j$ if and only if $2j\equiv j\pmod{q-1}$, i.e., $j=0$.
This means that if we consider double cusp forms, the corresponding matrices for the particular blocks considered
in Theorem \ref{ThmAntidiagonal} are exactly the same.
\end{rem}

\begin{exe} \label{Exn=j+2o3}
{\em The following matrices show that the bound in Theorem \ref{ThmAntidiagonal} is sharp. For $q=8$, $j=3, 6$ we have }
\[ M(3,5,8)=\left( \begin{array}{ccccc}
1 & 0 & 0 & 0 & 0 \\
0 & 0 & 0 & 1 & 0 \\
0 & 0 & 1 & 0 & 0 \\
0 & 1 & 0 & 0 & 0 \\
1 & 0 & 0 & 0 & 0 \end{array}\right) \ {\rm and} \
M(6,8,8)=\left(\begin{array}{cccccccc} 1 & 1 & 1 & 1 & 1 & 1 & 1 & 0\\
0 & 0 & 0 & 0 & 0 & 0 & 1 & 0\\
0 & 0 & 0 & 0 & 0 & 1 & 0 & 0\\
0 & 0 & 0 & 0 & 1 & 0 & 0 & 0\\
0 & 0 & 0 & 1 & 0 & 0 & 0 & 0\\
0 & 0 & 1 & 0 & 0 & 0 & 0 & 0\\
0 & 1 & 0 & 0 & 0 & 0 & 0 & 0\\
1 & 0 & 0 & 0 & 0 & 0 & 0 & 0 \end{array}\right)\,.\]
{\em Another example in odd characteristic: for $q=9$ and $j=3$ we have }
 \[ M(3,5,9)=\left(\begin{array}{ccccc}
1 & -1 & 0 & -1 & 0 \\
0 & 0 & 0 & -1 & 0 \\
0 & 0 & -1 & 0 & 0 \\
0 & -1 & 0 & 0 & 0 \\
-1 & 0 & 0 & 0 & 0
\end{array}\right)\,.\]
{\em We shall deal with the $n=j+2$ case in Theorem \ref{Thmn=j+2}, we shall also provide all details on $M(3,5,9,t)$ in
Example \ref{Ex359}.}
\end{exe}

\subsubsection{Antidiagonal blocks and newforms}\label{SecAntidiagonalNewforms}
The bound for Theorem \ref{ThmAntidiagonal} basically determines the maximal $k$ for which $\dim(S^1_{k,m}(\G_0(1)))=0$.
To see this, with notations as in \cite{Cor}, consider the modular form $g\in M_{q-1,0}(GL_2(A))$ and the cusp form
$h\in M_{q+1,1}(GL_2(A))$ which generate $M_{k,m}(GL_2(A))$, i.e., such that
$M_{k,m}(GL_2(A))\simeq \C_\infty [g,h]$ (see \cite[Proposition 4.6.1]{Cor}) where the polynomial ring
is intended doubly graded by weight and type.
When $n\leqslant j+1< q-1$, using \cite[Proposition 4.3]{Cor} and denoting by $\lfloor \cdot \rfloor$ the floor function,
one gets
\begin{align*} \dim_{\C_\infty} M_{k,m}(GL_2(A)) & = 1+\left\lfloor \frac{k-(j+1)(q+1)}{q^2-1} \right\rfloor \\
 & = 1+\left\lfloor \frac{n-2-j}{q+1} \right\rfloor\,.
 \end{align*}
Since $-1 < \frac{n-2-j}{q+1} <0$, $ \dim_{\C_\infty} M_{k,m}(GL_2(A))=0$. \\
When $n\leqslant j+1$ and $j+1=q-1$
\begin{align*}
\dim_{\C_\infty} M_{k,m}(GL_2(A)) & = 1+ \left\lfloor \frac{k}{q^2-1} \right\rfloor\,.
\end{align*}
Seeing that $0<\frac{k}{q^2-1}=\frac{n+1}{q+1}<1$, we end up with a one dimensional space generated by $g^{q}$
(obviously not cuspidal).\\
It remains to show that such a $k$ is maximal. For this just consider $n=j+2$. With computations as above it is easy
to see that:
\begin{enumerate}
\item if $j+1<q-1$, then $\dim_{\C_\infty} M_{k,m}(GL_2(A))=1$ generated by $\{h^{j+1}\}$;
\item if $j+1=q-1$, then $\dim_{\C_\infty} M_{k,m}(GL_2(A))=2$ generated by $\{h^{j+1},g^{q+1} \}$.
\end{enumerate}

Anyway it is quite easy to prove that all the eigenforms involved in an antidiagonal block are newforms (we mention this because it
holds even if the whole matrix is not antidiagonal). Indeed the Fricke action on cocycles (for any $i$) is given by
\begin{align*}
\c_i^{Fr}(\e)(X^\ell Y^{k-2-\ell}) & = {\matrix{0}{-1}{t}{0}}^{-1}\c_i
\left(\matrix{0}{-1}{t}{0}\matrix{0}{1}{1}{0}\right)(X^\ell Y^{k-2-\ell}) \\
\ & = {\matrix{0}{\frac{1}{t}}{-1}{0}}\c_i
\left(\matrix{-1}{0}{0}{1}\matrix{1}{0}{0}{t}\right)(X^\ell Y^{k-2-\ell}) \\
\ & = {\matrix{0}{\frac{1}{t}}{1}{0}}\c_i(e)(X^\ell Y^{k-2-\ell})
= (-1)^m t^{m-1-\ell} \c_i(\e) (X^{k-2-\ell}Y^\ell)  \,.
\end{align*}
So
\begin{equation}\label{EqFrCoc}
\c_i^{Fr}=(-1)^mt^{i+1+m-k} \c_{k-2-i}
\end{equation}
(note that $\c_i$ and $\c_{k-2-i}$ correspond to ``symmetric'' columns in the matrices above).

Recall that to use our formulas for $U_t$ on cocycles in the setting of Section \ref{SecSlopes} we need
to divide by $t^{k-m}$. Therefore equations \eqref{EqTr} and \eqref{EqTr'} translate into
\begin{equation}\label{EqTrCoc}
Tr(\c_i)=\c_i+\frac{t^{k-2m}}{t^{k-m}}U_t(\c_i^{Fr})=\c_i+(-1)^m t^{i+1-k}U_t(\c_{k-2-i})
\end{equation}
and
\begin{equation}\label{EqTr'Coc}
Tr'(\c_i)=\c_i^{Fr}+\frac{1}{t^{k-m}}U_t(\c_i)=(-1)^m t^{i+1+m-k}\c_{k-2-i}+t^{m-k}U_t(\c_i)\,.
\end{equation}

Now move to our matrices $M(j,n,q,t)$ with $k=2j+2+(n-1)(q-1)$ and $n\leqslant j+1$. In this antidiagonal setting
one has, for $1\leqslant h\leqslant n$,
\[ U_t(\c_{j+(h-1)(q-1)})= (-1)^{j+2}t^{j+1+(h-1)(q-1)}\c_{k-2-j-(h-1)(q-1)} \,.\]
Then
\begin{align*}
Tr(\c_{j+(h-1)(q-1)}) & =\c_{j+(h-1)(q-1)} \\
\ & +(-1)^m t^{j+(h-1)(q-1)+1-k}(-1)^{j+2}t^{k-2-j-(h-1)(q-1)+1}\c_{k-2-k+2+j+(h-1)(q-1)}  \\
\ & = \c_{j+(h-1)(q-1)}(1+(-1)^{m+j+2})=0\,,
\end{align*}
where the last equality comes from the fact that $j\equiv m-1 \pmod{q-1}$. Formula \eqref{EqFrCoc} readily implies
$Tr'(\c_{j+(h-1)(q-1)})=0$ as well, therefore all these cocycles are newforms (in general, whenever one has an
antidiagonal block inside some $M(j,n,q,t)$ the associated cocycles are newforms). Obviously the eigenvectors
are newforms too: they are of the form
\[ (-1)^{j+1}t^{\frac{k}{2}-j-(h-1)(q-1)-1}\c_{j+(h-1)(q-1)}\pm \c_{k-2-j-(h-1)(q-1)} \]
if $n$ is even, when $n$ is odd one simply adds $\c_{\frac{k-2}{2}}$.

\subsection{Diagonalizability of $M(j,n,q)$: case $j=0$}
The previous Theorem \ref{ThmAntidiagonal} leaves out many cases especially for small values of $j$ because of the
bound $n\leqslant j+1$.\\
For the particular case of $j=0$ (where the above theorem does not apply at all) we have, for example,
\[ M(0,4,q)=\left(\begin{array}{cccc} 1 & 0 & 0 & 0 \\
1 & 0 & 1 & 1\\
1 & 1 & 0 & 1\\
1 & 0 & 0 & 0 \end{array}\right) \quad{\rm for\ any\ even\ } q  \]
and
\[ M(0,6,q)=\left(\begin{array}{cccccc} 1 & 0 & 0 & 0 & 0 & 0\\
1 & 0 & 0 & 0 & 1 & 1\\
1 & 0 & 0 & 1 & 0 & 1\\
1 & 0 & 1 & 0 & 0 & 1\\
1 & 1 & 0 & 0 & 0 & 1\\
1 & 0 & 0 & 0 & 0 & 0
\end{array}\right) \quad{\rm for\ any\ even\ }q\geqslant 4 \,.\]

Other numerical computations led us to think that this could be generalized in the following

\begin{thm}\label{Thmj=0}
Let $n\in\N$ with $n\geqslant 2$ and $j=0$. Then, for all $n\leqslant q+2$,
the matrix $M(0,n,q)$ has the following entries
\begin{enumerate}
\item {$m_{a,1}=1$ for $1\leqslant a\leqslant n$;}
\item {$m_{a,b}=0$ for $1\leqslant a\leqslant n-2$, $2\leqslant b\leqslant \frac{n}{2}$ (or $\frac{n+1}{2}$
depending on the parity of $n$) and $b<n+1-a$,}
\end{enumerate}
i.e.,
\[ M(0,n,q)=\left( \begin{array}{cccccc}
1 & 0 & \cdots & \cdots & 0 & 0 \\
1 & 0 & \cdots & 0 & 1 & -1 \\
\vdots & \vdots & \ & \udots & 0 & \vdots \\
\vdots & 0 & \udots &  & \vdots & \vdots \\
1 & 1 & 0 & \cdots & 0 & -1 \\
1 & 0 & \cdots & \cdots & 0 & 0 \end{array}\right) \ \footnote{Note
the central antidiagonal block which has to come from newforms as
mentioned above.}\,.\]
\end{thm}

\begin{proof}
Directly from Section \ref{SecSymmetry} (in particular statements 3, 4, 5, 6 and 7) we have
\begin{itemize}
\item $m_{a,n+1-a}=1$ (the antidiagonal), for $\frac{n}{2}+1\leqslant a$ (or $\frac{n+1}{2}\leqslant a$);
\item $m_{a,b}=0$, for $b> n+1-a$, $\frac{n}{2}+1\leqslant a$ (or $\frac{n+1}{2}\leqslant a$) (below the antidiagonal);
\item if $n$ is odd, $m_{a,\frac{n+1}{2}}=0$ for any $a\neq \frac{n+1}{2}$ (central column).
\end{itemize}
Hence we limit ourselves to elements above the antidiagonal in the first $\frac{n}{2}$ (or $\frac{n-1}{2}$) columns.

In the first column, the elements $m_{a,1}$, $1\leqslant a\leqslant n-1$ are the following
\begin{itemize}
\item $m_{1,1}=\binom{(n-1)(q-1)}{0}=1$;
\item for any $a\geqslant 2$, $m_{a,1}=-\left[\binom{(n-a)(q-1)}{(n-1)(q-1)}-\binom{(n-a)(q-1)}{0}\right]=1$.
\end{itemize}

In the first row the elements $m_{1,b}$, $2\leqslant b\leqslant \frac{n}{2}$ (or $\leqslant \frac{n-1}{2}$) are
\[ m_{1,b}=-\left[ \binom{(n-1)(q-1)}{(n-b)(q-1)}-\binom{(n-1)(q-1)}{(b-1)(q-1)}\right]=0\,.\]

We are now going to check the entries $m_{a,b}$ with $2\leqslant a\leqslant n-2$, $b< n+1-a$ and
$2\leqslant b\leqslant \frac{n}{2}$ (or $\leqslant \frac{n-1}{2}$). We begin with elements on the diagonal.
Write $m_{a,a}$ as
\[ m_{a,a}=\binom{(n-a)(q-1)}{(a-1)(q-1)}=\binom{(n-a-1)q+q+a-n}{(a-2)q+q-a+1}\,.\]
By our hypotheses on $n$ and $q$ and our current bounds on $a$, we have $0\leqslant q+a-n,q+1-a <q$. Hence
we can write the $p$-adic expansions as
\begin{align*}
(n-a-1)q+q+a-n & = \a_0+\a_1 p+\dots+\a_{r-1}p^{r-1}+ (n-a-1)p^r\\
(a-2)q+q-a+1& = \b_0+\b_1 p+\dots+\b_{r-1}p^{r-1}+(a-2)p^r\,.
\end{align*}
If there exists $i$ such that $\a_i<\b_i$, then $m_{a,a}=0$. If for any $i$ one has $\a_i\geqslant\b_i$,
then $q+a-n\geqslant q-a+1$ and this means that $a\geqslant \frac{n+1}{2}$. But this is outside our range
(and we already know $m_{\frac{n+1}{2},\frac{n+1}{2}}=1$), hence $m_{a,a}=0$ in the first columns.\\
Outside the diagonal we write the general entry as
\[ m_{a,b}=-\left[\binom{(n-a-1)q+q+a-n}{(n-b-1)q+q+b-n}-\binom{(n-a-1)q+q+a-n}{(b-2)q+q-b+1}\right]\,. \]
As before the hypotheses on $n$ and $q$ and the current bounds on $a$ and $b$ allow us to write the $p$-adic expansions as
\begin{align*}
(n-a-1)q+q+a-n & =\a_0+\a_1 p+\dots+\a_{r-1}p^{r-1}+(n-a-1)p^r;\\
(b-2)q+q-b+1 & = \g_0+\g_1 p+\dots+\g_{r-1}p^{r-1}+(b-2)p^r;\\
(n-b-1)q+q+b-n & = \delta_0+\delta_1 p+\dots+\delta_{r-1}p^{r-1}+(n-b-1)p^r\,.
\end{align*}
If there exists $i$ such that $\a_i<\g_i$, then the first binomial coefficient is zero.
If for any $i$ one has $\a_i\geqslant\g_i$, then $a+q-n\geqslant q-b+1$, i.e., $b\geqslant n+1-a$
which is outside our current range.\\
If there exists $i$ such that $\a_i<\delta_i$, then the second binomial coefficient is zero.
If for any $i$ one has $\a_i\geqslant\delta_i$, then $q+a-n\geqslant q+b-n$, i.e., $a\geqslant b$
and so $(n-a)(q-1) < (n-b)(q-1)$. Therefore, the second binomial coefficient is 0 as well.
\end{proof}

\begin{cor}
Let $n\in\N$ and $j=0$. Then, for all $2\leqslant n\leqslant q+2$, the slopes are 1 and $\frac{k}{2}$.
Moreover the matrix $M(0,n,q,t)$ is diagonalizable if and only if $q$ is odd or $q$ is even and $n\leqslant 3$.
\end{cor}

\begin{proof}
Putting back the powers of $t$ in the coefficient matrix above one has
\[ \det(M(0,n,q,t)-X\!\!\cdot{\bf Id}_n)=\left\{ \begin{array}{ll} (X^2-tX)(X^2-t^k)^{\frac{n}{2}-1} & {\rm if}\ n\ {\rm is\ even} \\
\ & \\
(X^2-tX)(X^2-t^k)^{\frac{n-3}{2}}(-X+t^{\frac{k}{2}}) & {\rm if}\ n\ {\rm is\ odd} \end{array} \right. \,,\]
i.e., slopes are $1=j+1$ and $\frac{k}{2}$. Diagonalizability basically depends on the central antidiagonal block, so
the final statement is straightforward: as seen in Section \ref{SecAntidiagonalNewforms}, for $q$ odd we have $n-2$ newforms
of eigenvalues $\pm t^{\frac{k}{2}}$ and 2 oldforms of eigenvalues 0 and $t$.
\end{proof}

\begin{rem}
We observe that in the class $C_0$ we will always find the cocycles $\c_0$ and $\c_{k-2}=\c_{(n-1)(q-1)}$. Deleting
first and last columns and first and last rows of the matrix in Theorem \ref{Thmj=0} we will get the corresponding
matrix for double cusp forms, which is antidiagonal.
\end{rem}

\begin{exe}
{\em The following matrices show that the bound on $n$ is as accurate as possible.
Indeed, for $j=0$ and $n=q+3$ the form is different from the one in the above theorem.
\[ M(0,5,2)=\left( \begin{array}{ccccc}
1 & 0 & 0 & 0 & 0\\
1 & 1 & 0 & 0 & 1\\
1 & 0 & 1 & 0 & 1\\
1 & 1 & 0 & 0 & 1\\
1 & 0 & 0 & 0 & 0
\end{array}\right) \quad
M(0,7,4)=\left( \begin{array}{ccccccc}
1 & 0 & 0 & 0 & 0 & 0 & 0 \\
1 & 1 & 0 & 0 & 0 & 0 & 1 \\
1 & 0 & 0 & 0 & 1 & 0 & 1 \\
1 & 0 & 0 & 1 & 0 & 0 & 1 \\
1 & 0 & 1 & 0 & 0 & 0 & 1 \\
1 & 1 & 0 & 0 & 0 & 0 & 1 \\
1 & 0 & 0 & 0 & 0 & 0 & 0
\end{array}\right)\]
\[ M(0,6,3)=\left(\begin{array}{cccccc}
1 & 0 & 0 & 0 & 0 & 0\\
1 & 1 & 0 & 0 & 0 & -1\\
1 & 0 & 0 & 1 & 0 & -1\\
1 & 0 & 1 & 0 & 0 & -1\\
1 & 1 & 0 & 0 & 0 & -1\\
1 & 0 & 0 & 0 & 0 & 0\\
\end{array}\right) \quad
M(0,12,9)= \left(\begin{array}{cccccccccccc}
1 & 0 & 0 & 0 & 0 & 0 & 0 & 0 & 0 & 0 & 0 & 0\\
1 & 1 & 0 & 0 & 0 & 0 & 0 & 0 & 0 & 0 & 0 & -1\\
1 & 0 & 0 & 0 & 0 & 0 & 0 & 0 & 0 & 1 & 0 & -1\\
1 & 0 & 0 & 0 & 0 & 0 & 0 & 0 & 1 & 0 & 0 & -1\\
1 & 0 & 0 & 0 & 0 & 0 & 0 & 1 & 0 & 0 & 0 & -1\\
1 & 0 & 0 & 0 & 0 & 0 & 1 & 0 & 0 & 0 & 0 & -1\\
1 & 0 & 0 & 0 & 0 & 1 & 0 & 0 & 0 & 0 & 0 & -1\\
1 & 0 & 0 & 0 & 1 & 0 & 0 & 0 & 0 & 0 & 0 & -1\\
1 & 0 & 0 & 1 & 0 & 0 & 0 & 0 & 0 & 0 & 0 & -1\\
1 & 0 & 1 & 0 & 0 & 0 & 0 & 0 & 0 & 0 & 0 & -1\\
1 & 1 & 0 & 0 & 0 & 0 & 0 & 0 & 0 & 0 & 0 & -1\\
1 & 0 & 0 & 0 & 0 & 0 & 0 & 0 & 0 & 0 & 0 & 0
\end{array}\right)\,.\]
Adding the powers of $t$, they all give rise to diagonalizable matrices, except $M(0,7,4)$ because of the antidiagonal
block of dimension 3.}
\end{exe}

There seems to be a pattern for matrices of type $M(0,q+3,q)$ but this kind of case by case analysis does not seem to
be useful for large values of $k$. We give a final example in the next section and then provide (in Section \ref{SecTables})
some results arising from a more extensive computer search on slopes and eigenvalues.

\subsection{The case $n=j+2$ and diagonalizability of matrices of dimension at most 4}\label{Sec=j+2&j+3}
We put here one more result which generalizes the matrices appearing in Example \ref{Exn=j+2o3}.
We decided to add this case because the form of the matrix is still acceptable (and the characteristic polynomial
is easy to compute) and, together with the previous results and the final example of $M(1,4,q,t)$ (see below),
will establish the $n=4$ case (for which the indices $j=1,2$ fell outside the bounds of the previous theorems).

\begin{thm}\label{Thmn=j+2}
Let $j\geqslant 2$ be even and let $n=j+2$, then the matrix has the form
\[ M(j,j+2,q)=\left( \begin{array}{cccccc}
1 & m_{1,2} & \cdots & \cdots & (-1)^{j+1}m_{1,2} & 0 \\
0 & 0 & \cdots & 0 & (-1)^{j+2} & \vdots \\
\vdots & \vdots & \ & \udots & 0 & \vdots \\
\vdots & 0 & \udots &  & \vdots & \vdots \\
0 & (-1)^{j+2} & 0 & \cdots & 0 & \vdots \\
(-1)^{j+2} & 0 & \cdots & \cdots & 0 & 0 \end{array}\right) \]
and $M(j,j+2,q,t)$ has eigenvalues 0, $t^{j+1}$ and $\pm t^{\frac{k}{2}}$.
\end{thm}

\begin{proof}
We immediately note that if $M(j,j+2,q)$ has the above form, then the matrix $M(j,j+2,q,t)$ has characteristic polynomial
\[ \det (M(j,j+2,q,t)-X\!\!\cdot{\bf Id}_n)= (X^2-t^{j+1}X)\cdot \left\{ \begin{array}{ll}
(X^2-t^k)^{\frac{n-2}{2}} & {\rm if\ }n\ {\rm is\ even} \\
\ & \\
(X^2-t^k)^{\frac{n-3}{2}}(-X+(-1)^{j+2}t^{\frac{k}{2}}) & {\rm if\ }n\ {\rm is\ odd} \end{array}\right. ,\]
hence the eigenvalues are 0, $t^{j+1}$ and $\pm t^{\frac{k}{2}}$.

\noindent On the first column we have
\[ m_{1,1}=(-1)^{j+2}\binom{j+(j+1)(q-1)}{j}=(-1)^{j+2}\binom{jq+q-1}{j}=1 \]
(by Lucas's Theorem \ref{KummerThm} \footnote{Recall that $\binom{q-1}{j}\equiv(-1)^j\pmod{p}$
for $0\leqslant j\leqslant q-2$.
}), and
\[ m_{a,1}=-\left[\binom{(j+2-a)q+a-2}{jq+q-1}+(-1)^{j+1}\binom{(j+2-a)q+a-2}{j}\right]=0 \quad 2\leqslant a\leqslant n-1=j+1 \]
(as usual the first binomial can be nonzero only if $a-2=q-1$, and the second one can be nonzero only if
$a-2\geqslant j$ and these conditions cannot hold).

\noindent In a similar way one proves that in the columns between 2 and $\frac{n}{2}$ (or $\frac{n+1}{2}$),
above the antidiagonal one has
\[ m_{a,a}=(-1)^{j+2}\binom{(j+2-a)q+a-2}{(a-1)q+j-a+1} = 0 \quad 2\leqslant a\leqslant \frac{n}{2}\ {\rm (or\ }\frac{n-1}{2}{\rm)}\]
and
\[ m_{a,b}=-\left[\binom{(j+2-a)q+a-2}{(j+2-b)q+b-2}+(-1)^{j+1}\binom{(j+2-a)q+a-2}{(b-1)q+j-b+1}\right]=0 \]
for $2\leqslant b\leqslant \frac{n}{2}$ (or $\frac{n+1}{2}$) and $2\leqslant a\leqslant n-b$
(note that the coefficients of the first line, except $m_{1,1}$, were not computed but they do not affect
the characteristic polynomial).
\end{proof}

Diagonalizability is not immediately clear because of the first row. We give one example in low dimension.

\begin{exe}\label{Ex359}
{\em We give full details for the matrix
\[ M(3,5,9,t)=\left(\begin{array}{ccccc}
t^4  & -t^{12} & 0       & -t^{28} & 0 \\
0    & 0       & 0       & -t^{28} & 0 \\
0    & 0       & -t^{20} & 0       & 0 \\
0    & -t^{12} & 0       & 0       & 0 \\
-t^4 & 0       & 0       & 0       & 0 \end{array} \right) \]
with $k=40$, $m=4$ and $C_3=\{\c_3,\c_{11},\c_{19},\c_{27},\c_{35}\}$.
The eigenvalues are 0, $t^4$, $t^{20}$ and $-t^{20}$ (the last one with multiplicity 2).

\noindent Oldforms are easy: $\c_{35}$ is an oldform of eigenvalue 0 and $\c_3-\c_{35}$ is an oldform of eigenvalue $t^4$.
Note that (again with the notations of \cite{Cor}) the eigenvalue $t^4$ comes from the (old) cusp form $h^4$,
which is a generator for $S^1_{40,4}(GL_2(A))$ (see \cite[Proposition 4.3 and Proposition 4.6.2]{Cor} and
\cite[Corollary 7.6]{G2}).

\noindent Using the formulas of Section \ref{SecAntidiagonalNewforms} one checks that\begin{itemize}
\item $\c_{19}$ and
\[ \c:=(t^{24}+t^{16})\c_3+(t^{24}+t^8)\c_{11}+(1+t^{16})\c_{27}+(1+t^8)\c_{35} \]
represent (independent) newforms of eigenvalue $-t^{20}$ (moreover the Fricke involution acts on them with multiplication
by $t^{-16}$);
\item the cocycle
\[ \c':=(t^{24}-t^{16})\c_3+(t^{24}-t^8)\c_{11}+(1-t^{16})\c_{27}+(1-t^8)\c_{35} \]
represents a newform of eigenvalue $t^{20}$.
\end{itemize} }
\end{exe}

As a final example we write down $M(1,4,q,t)$ (for any $q>2$) for two reasons:\begin{itemize}
\item it provides an example for all $q$ hinting at the fact that diagonalizability depends more on the
symmetry of the matrices than on $q$ (except of course for the $q$ even case);
\item it is outside our parameters (indeed $n=j+3$ here) and it completes the case $n=4$, indeed Theorems
\ref{ThmAntidiagonal}, \ref{Thmj=0} and \ref{Thmn=j+2} together with this example show that $M(j,n,q,t)$
is diagonalizable for any $n\leqslant 4$ for odd $q$. In characteristic 2 $M(j,n,q,t)$, with $2\leqslant n\leqslant 4$,
is diagonalizable only for $n=2$ and $j=0$ or $n=3$ and $j=0,1$.
\end{itemize}

Using Lemma \ref{KummerThm} we have
\[ M(1,4,q,t)=\left( \begin{array}{cccc} 2t^2 & -2t^{2+(q-1)} & -2t^{2+2(q-1)} & t^{2+3(q-1)} \\
t^2 & -t^{2+(q-1)} & -2t^{2+2(q-1)} & t^{2+3(q-1)} \\
0 & -t^{2+(q-1)} & 0 & 0 \\
-t^2 & 0 & 0 & 0 \end{array} \right) \] whose characteristic
polynomial is $X(X+t^{2+(q-1)}-2t^2)(X^2-t^k)$ (where
$k=4+3(q-1)\,$). It is obviously diagonalizable in odd
characteristic because it has distinct eigenvalues (slopes are 2 and
$\frac{k}{2}$) and non diagonalizable in characteristic 2 because of
the inseparable eigenvalue $t^{\frac{k}{2}}$. Note that $2t^2 -t^{2+(q-1)}$ has to be
the eigenvalue associated to the oldform $h^2g$ (notations as in \cite{Cor} and $q>2$).

\section{Links to tables and hints for future research}\label{SecTables}
Here we collect some speculations and conjectures which come from data we have been gathering while trying to understand how
eigenvalues behave when changing weight and type. \\
First of all we warn the reader that eigenvalues and slopes appearing here depend on the normalization used in Section
\ref{SecGamma1}; so, for example, we will find slopes $\frac{k}{2}$ instead of $m-\frac{k}{2}$ for newforms.

With the software Mathematica (\cite{W}) we implemented formula \eqref{Ttcj} to calculate the cha\-rac\-te\-ri\-stic
polynomials of about 250 matrices associated with $U_t$ for $q$ even up to $2^8$ and odd $k\geqslant q+3$ (it is easy
to see that matrices are diagonalizable for $k\leqslant q+2$ because almost all blocks are of dimension 1).
Those polynomials were computed when we were looking for inseparable eigenvalues, that is why only odd $k$ were taken into account.
Table 1 presents the collection of these polynomials and it can be
downloaded at the following link \href{https://docs.google.com/viewer?a=v&pid=sites&srcid=ZGVmYXVsdGRvbWFpbnxtYXJpYXZhbGVudGlubzg0fGd4OjI3OGIwNDUyMmVmODI2NjA}{[click here]}.
Furthermore, in Table 2 we present the characteristic polynomials of some blocks $C_j$ arising from $\Gamma_0(t)$ and computed using formula \eqref{spam} (to download click \href{https://docs.google.com/viewer?a=v&pid=sites&srcid=ZGVmYXVsdGRvbWFpbnxtYXJpYXZhbGVudGlubzg0fGd4OjJkMTVjMWU1MDdmMmIwMDM}{here}).\\
Finally, using Pari/GP (\cite{PARI2}) we computed the characteristic polynomials and slopes of $\T_t$ acting on
$M_{k}(GL_2(A))$ for $q=2$ and $4\leqslant k\leqslant 98$ in Table 3 (to download \href{https://docs.google.com/viewer?a=v&pid=sites&srcid=ZGVmYXVsdGRvbWFpbnxtYXJpYXZhbGVudGlubzg0fGd4Ojc3MzY4NzBlOGRjNGE5OQ}{click here}).
Please note that such data are not related to cuspidal forms only. In particular, the first slope is always
equal to the weight and it is not coming from a cuspidal form. This should be taken into account when looking at the characteristic polynomial.

The whole collection of data lead to the following remarks/speculations.
\begin{enumerate}
\item Surprisingly, we found fractional slopes (hence inseparable eigenvalues) also at level one.
For example, look at weight $k=9$ to find the slope $\frac{5}{2}$ and $k=15$ for the slope $\frac{7}{2}$.
This leads to a natural question: can we find inseparable eigenvalues in odd characteristic as well?\\
In order to find an answer we computed the characteristic polynomials associated
with $U_t$ acting on $S^1_{k}(\Gamma_1(t))$ for $q=3$ and $6\leqslant k \leqslant 42$ and also for even $k$ up to $62$ in Table 4
(to download click \href{https://docs.google.com/viewer?a=v&pid=sites&srcid=ZGVmYXVsdGRvbWFpbnxtYXJpYXZhbGVudGlubzg0fGd4OjZkMzU3ZTIwMzQxNjVmN2Q}{here}).
We included the odd weights (where there are no $\Gamma_0(t)$-invariant forms) as far as possible for completeness, but factorizing multivariate polynomials in positive characteristic requires plenty of effort and time, even with a good
software. We could go on for even weights because, by results of Section \ref{SecSlopes}, we know we could always try to factor out $(x^2\pm t^{k/2})$.
The presence of such factor confirms our impression that the action of $U_t $ on newforms is ``morally'' antidiagonal. 
\item The characteristic polynomials for level one forms (Table 3) have non trivial constant term. For diagonalizability reasons
(as already mentioned after Remark \ref{RemOld}), we believe that $\T_t$ has never eigenvalue 0 at level 1. This is confirmed by 
the results of Table 4 where, for even $k$, the multiplicity of the eigenvalue 0 is always equal to the sum of the multiplicities of the old eigenvalues.
However, it is natural to ask if 0 can be an eigenvalue for $\T_t$ if the level $\mathfrak{m}$ is coprime with $t$.
\item The correspondence between multiplicity of 0 and number old eigenavalues mentioned above seems to suggest that there should be trivial intersection between 
new and oldforms.
\item Last, we would like to observe that something as \cite[Conjecture 1]{GM1} can be considered in our setting as well.
Let us illustrate the details.  Let $g_k\in \F_2[t][X]$, for $k\in\Z$, be the characteristic polynomial of $U_t$
acting on $S^1_k(\Gamma_1(t))$.  If $\alpha\in \Q$, then let $d(k,\alpha)$ denote the number of roots of $g_k$ in
$\overline{\F_2(t)}$ which have $t$-adic valuation equal to $\alpha$.
In our context, Gouv\^ea-Mazur's Conjecture 1 should look like

\begin{conj} If $k_1,k_2\in\Z$ are both at least $2\alpha+2$ and $k_1\equiv k_2\pmod {2^{n-1}}$ for
some $n\geqslant \alpha$, then $d(k_1,\alpha) = d(k_2,\alpha)$.
\end{conj}

\noindent For example, let us cast a glance over the following data.
\begin{itemize}
\item[$\circ$] {Let $\alpha=4$ and consider $8\leqslant k\leqslant 98$. We have that:
\begin{enumerate}
\item {$d(k,4)=1$ for $k\equiv 4 \pmod {2^3}$;}
\item {$d(k,4)=3$ for $k\equiv 0\pmod {2^3}$;}
\item {$d(k,4)=0$ for $k\not\equiv 0,4 \pmod {2^3}$.}
\end{enumerate}}
\item[$\circ$] {Let $\alpha=5/2$ and consider $8\leqslant k\leqslant 98$. We have that:
\begin{enumerate}
\item {$d(k,5/2)=2$ for $k\equiv 1\pmod {2^2}$;}
\item {$d(k,5/2)=0$ for $k\not\equiv 1\pmod {2^2}$.}
\end{enumerate}}
\item[$\circ$] {Let $\alpha=8$ and consider $18\leqslant k\leqslant 98$. We have that:
\begin{enumerate}
\item {$d(k,8)=5$ for $k\equiv 0 \pmod {2^4}$;}
\item {$d(k,8)=1$ for $k\equiv 8\pmod {2^4}$;}
\item {$d(k,8)=0$ for $k\not\equiv 0,8 \pmod {2^4}$.}
\end{enumerate}}

\end{itemize}

\end{enumerate}

\end{document}